\documentclass[3p,times]{elsarticle}

\journal{Advances in Mathematics}

\usepackage{amsmath,amsthm, amscd, amssymb, amsfonts, mathrsfs,graphics}

\usepackage[all]{xy}
\usepackage{color}

\usepackage{hhline}
\usepackage{eucal}

\usepackage[plainpages=false,pdfpagelabels,
hypertexnames=false]{hyperref}

\allowdisplaybreaks

\newtheorem{thm}{Theorem}

\newtheorem{lema}{Lemma}

\newtheorem{Cor}[lema]{Corollary}

\newtheorem{definition}[lema]{Definition}

\newtheorem{Rem}[lema]{Remark}
\newtheorem{Rems}[lema]{Remarks}

\newtheorem{convention}[lema]{Convention}

\newcommand{\GL}{\operatorname{GL}}

\newcommand{\ydstresd}{{}^{\ku^{\Sn_3}}_{\ku^{\Sn_3}}\mathcal{YD}}
\newcommand{\bB}{\mathbf{B}}

\newcommand{\cJ}{\mathcal{J}}
\newcommand{\cO}{\mathcal{O}}
\newcommand{\cS}{\mathcal{S}}

\newcommand{\cM}{\mathcal{M}}
\newcommand{\cC}{\mathcal{C}}

\newcommand{\Sn}{{\mathbb S}}

\newcommand{\B}{{\mathbb B}}

\newcommand{\xtop}{x_{top}}
\newcommand{\ytop}{y_{top}}
\newcommand{\xij}[1]{x_{(#1)}}
\newcommand{\yij}[1]{y_{(#1)}}

\newcommand{\DSn}{\D(\Sn_3)}
\newcommand{\oV}{\overline{V}}

\renewcommand{\_}[1]{_{\left( #1 \right)}}

\newcommand{\ot}{{\otimes}}
\newcommand{\ku}{\Bbbk}

\newcommand\fC{\mathsf{C}}

\newcommand\fE{\mathsf{E}}

\newcommand\fG{\mathsf{G}}

\newcommand\fJ{\mathsf{J}}

\newcommand\fL{\mathsf{L}}
\newcommand\fM{\mathsf{M}}
\newcommand\fN{\mathsf{N}}

\newcommand\fMsoc{\mathsf{M}_{\mathsf{soc}}}
\newcommand\fO{\mathsf{O}}
\newcommand\fP{\mathsf{P}}
\newcommand\fQ{\mathsf{Q}}
\newcommand\fR{\mathsf{R}}
\newcommand\fS{\mathsf{S}}
\newcommand\fT{\mathsf{T}}
\newcommand\fU{\mathsf{U}}
\newcommand\fV{\mathsf{V}}

\newcommand\fX{\mathsf{X}}

\newcommand\fa{\mathsf{a}}
\newcommand\fb{\mathsf{b}}
\newcommand\fc{\mathsf{c}}
\newcommand\fd{\mathsf{d}}
\newcommand\fe{\mathsf{e}}

\newcommand\fg{\mathsf{g}}

\newcommand\fj{\mathsf{j}}

\newcommand\fm{\mathsf{m}}
\newcommand\fn{\mathsf{n}}
\newcommand\fo{\mathsf{o}}
\newcommand\fp{\mathsf{p}}
\newcommand\fq{\mathsf{q}}
\newcommand\fr{\mathsf{r}}

\newcommand\ft{\mathsf{t}}
\newcommand\fu{\mathsf{u}}
\newcommand\fv{\mathsf{v}}

\newcommand\fy{\mathsf{y}}

\newcommand{\Z}{{\mathbb Z}}
\newcommand{\N}{{\mathbb N}}

\newcommand{\D}{{\mathfrak D}}
\newcommand{\BV}{{\mathfrak B}}

\newcommand{\ydh}{{}^H_H\mathcal{YD}}

\newcommand{\ydg}{{}^{\ku G}_{\ku G}\mathcal{YD}}

\newcommand{\Ind}{\operatorname{Ind}}

\newcommand\sgn{\operatorname{sgn}}
\newcommand\ad{\operatorname{ad}}

\newcommand\supp{\operatorname{Supp}}

\newcommand\mm[1]{\boldsymbol{|#1\rangle}}
\newcommand\e{\varepsilon}

\begin{document}

\begin{frontmatter}

\title{Verma and simple modules for quantum groups at non-abelian groups}

\author[rvt]{Barbara Pogorelsky}
\ead{barbara.pogorelsky@ufrgs.br}

\author[focal]{Cristian Vay\corref{cor1}}
\ead{vay@famaf.unc.edu.ar}

\cortext[cor1]{Corresponding author}

\address[rvt]{Instituto de Matem\'atica, Universidade Federal do Rio Grande do Sul, Av. Bento Goncalves 9500, Porto Alegre, RS, 91509-900, Brazil}

\address[focal]{Facultad de Matem\'atica, Astronom\'ia y F\'isica, Universidad Nacional de C\'ordoba, CIEM--CONICET,  Medina A\-llen\-de s/n,
Ciudad Universitaria, 5000 C\' ordoba, Rep\'ublica Argentina.}

\begin{abstract} The Drinfeld double $\D$ of the bosonization of a finite-dimensional Nichols algebra $\BV(V)$ over a finite non-abelian group $G$ is called a {\it quantum group at a non-abelian group}. We introduce Verma modules over such a quantum group  $\D$ and prove that a Verma module has simple head and simple socle. This provides two bijective correspondences between the set of simple modules over $\D$ and the set of simple modules over the Drinfeld double $\D(G)$. As an example, we describe the lattice of submodules of the Verma modules over the quantum group at the symmetric group $\Sn_3$ attached to the 12-dimensional Fomin-Kirillov algebra, computing all the simple modules and calculating their dimensions.
\end{abstract}


\end{frontmatter}

\section{Introduction}

The Drinfeld doubles of bosonizations of braided Hopf algebras over abelian groups, and their quotients by central group-likes, are known in the folklore as {\it quantum groups}. Such is the case of the {\it quantum enveloping algebra $U_{q}(\mathfrak{g})$} or the {\it small quantum group $u_{q}(\mathfrak{g})$}  \cite{drinfeld, lusztig}. These quantum groups have been intensely studied, both their intrinsic structures and their representation theories. However, to the best of our  knowledge, there is no research which contemplates non-abelian groups. The purpose of our work is to give a first step in this direction.

More precisely, let $G$ be a finite non-abelian group and $V$ a Yetter-Drinfeld module over $\ku G$ with finite-dimensional Nichols algebra $\BV(V)$ (we will work over an algebraically closed field $\ku$ of characteristic zero). We denote by $\D$ the Drinfeld double of the bosonization $\BV(V)\#\ku G$ and call it {\it ``quantum group at a non-abelian group''} -- recall that $U_{q}(\mathfrak{g})$ is called {\it quantum group at a root of $1$} if the indeterminate $q$ is specialized to a root of $1$ \cite{lusztig1}. In this work, we deal with the category of representations of $\D$. We use the methods coming from the theory of Lie algebras which  were also used in the study of quantum groups over abelian groups. We find similarities as well as differences between our results and their analogues in the context of Lie algebras or the mentioned quantum groups. To explain these similarities and differences, we first recall briefly the situation in those frameworks.

Assume that $U$ is either an enveloping algebra of a Lie algebra as in \cite[Chapter 7]{dixmier} or a quantum group as in \cite[Chapter 5]{Jantzen}, \cite[Chapter 3]{lusztig}; the reader can find all the details of the following exposition in these chapters. Roughly speaking, $U$ has a distinguished commutative and cocommutative Hopf subalgebra $U^0$. Hence
\begin{enumerate}\renewcommand{\theenumi}{\alph{enumi}}\renewcommand{\labelenumi}{(\theenumi)}
\item\label{intro:weights} the maximal spectrum of $U^0$ is an abelian group $T$. The elements of $T$, the algebra maps $U^0 \to \ku$, are called {\it weights}. The module corresponding to the weight $\alpha$ is denoted $\ku_{\alpha}$.
\item\label{intro:prod of weights} The  product of $T$ is implemented by tensoring.
That is, the tensor product of  $\ku_{\alpha}$ and $\ku_{\lambda}$ is  the module $\ku_{\alpha+\lambda}$.
\end{enumerate}

Also, $U$ admits a {\it triangular decomposition}, that means that there are subalgebras $U^{-}$ and $U^{+}$ such that the multiplication $U^{-}\,\ot U^0\,\ot U^{+}\rightarrow U$ gives a linear isomorphism. Indeed, $U$ is a $\Z$-graded algebra such that the degrees of $U^0$, $U^{-}$ and $U^{+}$ are zero, negative and positive, respectively. Let $U^{\geq0}$ be the subalgebra generated by $U^0$ and $U^{+}$. Given a weight $\lambda$, this can be seen as an $U^{\geq0}$-module by letting $U^+$ act trivially on it. We denote it again by $\ku_\lambda$.

Let $M$ be an $U$-module and $M_\lambda=\left\{m\in M\mid h\cdot m=\lambda(h)m\,\mbox{ for all }\,h\in U^0\right\}$ its {\it weight space} of weight $\lambda$. We restrict our attention to the $U$-modules which decompose as the direct sum of their weight spaces. For instance, $U$ regarded as a module with respect to the {\it adjoint action}. Then

\begin{enumerate}\renewcommand{\theenumi}{\alph{enumi}}\renewcommand{\labelenumi}{(\theenumi)}\addtocounter{enumi}{2}
\item\label{intro:decomp weights} $U_\alpha\cdot M_\lambda\subseteq M_{\alpha+\lambda}$ for all weights $\alpha$ and $\lambda$.
\end{enumerate}

An $U$-module $M$ is called a {\it highest-weight module} (of weight $\lambda$) if it is generated by an element $v\in M_\lambda$ such that $U^+v=0$. Notice that $M=Uv=U^-v$ by the triangular decomposition of $U$. The basic examples of highest-weight modules are the {\it Verma modules} $M(\lambda)=U\ot_{U^{\geq0}}\ku_\lambda$. These are essential in the study of the representation theory of $U$ because
\begin{enumerate}\renewcommand{\theenumi}{\alph{enumi}}\renewcommand{\labelenumi}{(\theenumi)}\addtocounter{enumi}{3}
\item\label{intro:verma} Every Verma module has a unique simple quotient and every simple $U$-module is a quotient of a unique Verma module.
\item\label{intro:verma free} Every Verma module is a free $U^-$-module of rank $1$.
\end{enumerate}
Other important features of the Verma modules can be found in the  references above.

Let us consider now a quantum group $\D$ at a finite non-abelian group $G$. The  role of $U^0$ shall be played by the Drinfeld double $\D(G)$ of $\ku G$. This is a semisimple but not commutative Hopf subalgebra of $\D$. We will see that $\D$ admits a triangular decomposition
$$
\D=\BV(V)\ot\D(G)\ot\BV(\oV)
$$
where $\oV$ denotes the dual object of $V$ in the category of $\D(G)$-modules and $\BV(\oV)$ is its Nichols algebra. In this setting the bosonization $\D^{\geq0}=\BV(\oV)\#\D(G)$ shall play the role of $U^{\geq0}$. We will calculate the commutation rules between the generators of $\D(G)$, $V$ and $\oV$, and deduce that $\D$ is a $\Z$-graded algebra with homogeneous spaces
$$
\D^{n}=\bigoplus_{n=j-i}\BV^{i}(V)\ot\D(G)\ot\BV^{j}(\oV).
$$

The classification of the simple $\D(G)$-modules is well-known, see for instance \cite[Subsection 3.1]{AG1}; unlike \eqref{intro:weights}, there are simple modules of dimension greater than one. The simple $\D(G)$-modules are parametrized by pairs $(\cO, \varrho)$, where $\cO$ is a conjugacy class in $G$ and $\varrho$ is an irreducible representation of the centralizer of a fixed $g\in \cO$. If $M(g,\varrho)$ denotes the corresponding simple $\D(G)$-module, cf. \eqref{eq:Mgrho}, then it becomes a $\D^{\geq0}$-module by letting  $\BV(\oV)$ act trivially on it. Therefore we can define the Verma modules for a quantum group at a non-abelian group as the induced modules
$$
\fM(g,\varrho)=\D\,\ot_{\D^{\geq0}}M(g,\varrho).
$$
Thus $\fM(g,\varrho)$ is a free $\BV(V)$-module of rank $\dim M(g,\varrho)=\#\cO_g\cdot\dim (U,\varrho)$, compare with \eqref{intro:verma free}.

Our main result asserts that \eqref{intro:verma} holds true in our context, {\it i.~e.} every Verma module $\fM(g,\varrho)$ has a unique simple quotient and every simple $\D$-module is a quotient of a unique Verma module, Theorem \ref{teo:bi con L}. Therefore we obtain a bijective correspondence
\begin{align*}
\biggl\{\mbox{Simple $\D(G)$-modules}\biggr\}&\leftrightsquigarrow\biggl\{\mbox{Simple $\D$-modules}\biggr\}\\
M(g,\varrho)\quad\quad\quad\quad&\leftrightsquigarrow\quad\quad\quad\quad\fL(g,\varrho)
\end{align*}
where $\fL(g,\varrho)$ denotes the head of $\fM(g,\varrho)$. Moreover, we prove that the socle $\fS(g,\varrho)$ of $\fM(g,\varrho)$ is simple what provides another bijective correspondence between the set of simple $\D(G)$-modules and the set of simple $\D$-modules, Theorem \ref{teo:bi con socle}. We also give a criterion to decide whether or not a Verma module is simple, Corollary \ref{cor:verma simple xtop}, and show that the socle and the head are related by
$$
\bigl(\fS(g,\varrho)\bigr)^*\simeq\fL(\hat{g}^*,\hat{\varrho}^*)
$$
where $M(\hat{g}^*,\hat{\varrho}^*)=\bigl(\BV^{top}(V)\ot M(g,\varrho)\bigr)^*$ and $\BV^{top}(V)$ is the homogeneous component of maximum degree of $\BV(V)$, Theorem \ref{teo:L dual S}; recall that $\BV^{top}(V)$ is one-dimensional.

In order to compute explicitly the simple $\D$-modules, we have to study the submodules of the Verma modules. This is done in the abelian case using the properties \eqref{intro:prod of weights} and \eqref{intro:decomp weights} among others, which allow to obtain remarkable results under certain general assumptions. Although the $\D$-modules decompose as the direct sum of simple $\D(G)$-modules, our situation is more complex because \eqref{intro:prod of weights} and \eqref{intro:decomp weights} do not hold true. Here the tensor product between simple $\D(G)$-modules is not necessarily simple and hence we have to know their fusion rules.

We give a general strategy to compute the highest-weight submodules of any $\D$-module $M$. We use that $\D$ is a $\Z$-graded $\D(G)$-module with respect to the adjoint action, which respects the triangular decomposition, and the fact that the action $\D\ot M\rightarrow M$ is a morphism of $\D(G)$-modules, \S\ref{subsec:algorithm}. We carry out this strategy to compute the simple modules in a concrete example in Section \ref{sec:S3} as we summarize below.

\subsection{A quantum group at the symmetric group \texorpdfstring{$\Sn_3$}{S3}}

The first genuine example of a finite-dimensional Nichols algebra over a non-abelian group is the Fomin-Kirillov algebra $\mathcal{FK}_3$ \cite{FK}. It is isomorphic to the Nichols algebra $\BV(V)$ of the Yetter-Drinfeld module $V=\ku\{\xij{12},\xij{23},\xij{13}\}$ over $\ku\Sn_3$. The action and coaction on $V$ are
$$
g\cdot\xij{ij}=\sgn(g)\,x_{g(ij)g^{-1}}\quad\mbox{and}\quad(\xij{ij})\_{-1}\ot(\xij{ij})\_{0}=(ij)\ot\xij{ij}
$$
for any transposition $(ij)$ and $g\in\Sn_3$ \cite{MS} where $\sgn:\Sn_3\rightarrow\{\pm1\}$ denotes the sign map.

Let now $\D$ be the Drinfeld double of $\BV(V)\#\ku\Sn_3$. As an algebra, $\D$ is generated by
\begin{align*}
-&\mbox{the generators of $\BV(V)$: }\quad\xij{12},\,\xij{23},\,\xij{13};\\
-&\mbox{the generators of $\DSn$: }\quad g,\,\delta_g\,\mbox{ for all } g\in\Sn_3;\\
-&\mbox{the generators of $\BV(\oV)$: }\quad\yij{12},\,\yij{23},\,\yij{13};
\end{align*}
we shall see that $V\simeq\oV$ as $\DSn$-modules. These elements are subjected to the next relations:
\begin{align*}
\tag{given by $\BV(V)$\,}\xij{ij}^2&=\xij{ij}\xij{ik}+\xij{jk}\xij{ij}+\xij{ik}\xij{jk}=\xij{ik}\xij{ij}+\xij{ij}\xij{jk}+\xij{jk}\xij{ik}=0,\\
\tag{given by $\BV(\oV)$\,}\yij{ij}^2&=\yij{ij}\yij{ik}+\yij{jk}\yij{ij}+\yij{ik}\yij{jk}=\yij{ik}\yij{ij}+\yij{ij}\yij{jk}+\yij{jk}\yij{ik}=0,
\end{align*}
\begin{align}
\tag{given by $\DSn$\,}\delta_h\,g=&g\,\delta_{g^{-1}hg},\\
\noalign{\smallskip}
\tag{given by the bosonizations}g\xij{ij}=\sgn(g)\,x_{g(ij)g^{-1}}g,\quad\quad&\delta_g\yij{ij}=\yij{ij}\delta_{(ij)g},
\end{align}
\begin{align}
\notag\delta_h\xij{ij}=\xij{ij}\delta_{(ij)h},&\quad\yij{ij}\,g=\sgn(g)\,g\,y_{g^{-1}(ij)g},\\
\noalign{\smallskip}
\tag{given by the definition of $\D$}\yij{ij}\xij{ij}+\xij{ij}\yij{ij}&=1+(ij)(\delta_{(ij)}-\delta_e),\\
\notag\yij{ik}\xij{ij}+\xij{ij}\yij{jk}&=(ij)(\delta_{(ik)}-\delta_{(ik)(ij)}),
\end{align}
for all transpositions $(ij)$, $(ik)$ and $g,h\in\Sn_3$.

On the other hand, the simple $\D$-modules are parametrized by the simple $\DSn$-modules according to our main result. Let $\sigma=(12)$ and $\tau=(123)$ be permutations in $\Sn_3$. Then $\cO_e$, $\cO_\sigma$ and $\cO_\tau$ are the conjugacy classes of $\Sn_3$ and $\DSn$ has eight non-isomorphic simple modules. Namely,
$$
M(e,+),\quad M(e,-),\quad M(e,\rho),\quad M(\sigma,+),\quad M(\sigma,-),\quad M(\tau,0),\quad M(\tau,1)\quad\mbox{and}\quad M(\tau,2).
$$
We recall the structures of them and their fusion rules in \S\ref{subsub:s3 sigma simples}--\S\ref{subsub:fusion rules}.

We compute the lattice of submodules of the corresponding Verma modules $\fM(g,\varrho)$ and classify the simple $\D$-modules. In particular, we prove that
\begin{itemize}
\item $\fL(e,+)\simeq M(e,+)$ as $\DSn$-modules and $\dim\fL(e,+)=1$, Corollary \ref{cor: l e mas cociente de T}.
\item $\fL(e,\rho)\simeq M(e,\rho)\oplus M(\sigma,+)\oplus M(\tau,0)$ as $\DSn$-modules and $\dim\fL(e,\rho)=7$, Corollary \ref{cor: l e rho cociente de VQN}.
\item $\fL(\tau,0)\simeq M(\tau,0)\oplus M(\sigma,+)\oplus M(e,\rho)$ as $\DSn$-modules and $\dim\fL(\tau,0)=7$, Corollary \ref{cor: l tau cero cociente de UPO}.
\item $\fL(\sigma,-)\simeq M(\sigma,-)\oplus M(\tau,1)\oplus M(\tau,2)\oplus M(\sigma,-)$ as $\DSn$-modules and $\dim\fL(\tau,0)=10$, Theorem \ref{teo:el latice de sigma menos}.
\item The Verma modules $\fM(e,-)$, $\fM(\tau,1)$, $\fM(\tau,2)$ and $\fM(\sigma,+)$ are simple, Theorem \ref{teo:los simples}. Their dimensions are $12$, $24$, $24$ and $36$, respectively. As $\DSn$-modules they are the tensor product of $\BV(V)$ with the associated simple $\DSn$-module.
\end{itemize}

We finish by pointing out other facts about these modules.
\begin{itemize}
\item[$\circ$] The head $\fL(g,\varrho)$ and the socle $\fS(g,\varrho)$ are isomorphic, Theorem \ref{teo:los simples}, Lemma \ref{le:socle sigma menos} and Corollaries \ref{cor: l e mas cociente de T}, except to
$$
\fL(\tau,0)\simeq\fS(e,\rho)\quad\mbox{and}\quad\fL(e,\rho)\simeq\fS(\tau,0),\quad\mbox{Corollary \ref{cor: l tau cero s e rho}}.
$$
\item[$\circ$] The simple $\D$-modules are self-duals except to
$$
\bigl(\fL(\tau,0)\bigr)^*\simeq\fL(e,\rho)\quad\mbox{and}\quad\bigl(\fL(e,\rho)\bigr)^*\simeq\fL(\tau,0)\quad\mbox{by Theorem \ref{teo:L dual S}}.
$$
\item[$\circ$] $\fM(\sigma,-)$ has submodules which are not homogeneous, Lemma \ref{le:T}, and its maximal submodule is not generated by highest-weight submodules, Theorem \ref{teo:el latice de sigma menos}.
\end{itemize}

\section*{Acknowledgments}
The first author was financially supported by Capes - Brazil. The second author was partially supported by ANPCyT-Foncyt, CONICET and Secyt (UNC).

This work was carried out in part during the visit of the first author to the University of C\'ordoba (Argentina). She would like to thank the Faculty of Mathematics, Astronomy and Physics for its warm hospitality and support. Both authors are grateful to Nicol\'as Andruskiewitsch for drawing their attention to this problem and also for so many suggestions. The second author thanks Vyacheslav Futorny for very useful conversations during his visit to the University of S\~ao Paulo under the framework of the MATH--AmSud program. We also thank the referee for the careful reading of our article and for providing constructive comments to improve the exposition of this paper.

\section{Preliminaries}
Through this work $\ku$ denotes an algebraically closed field of characteristic zero. The dual of a vector space $V$ will be denoted by $V^*$. If $v\in V$ and $f\in V^*$, then $\langle f,v\rangle$ denotes the evaluation of $f$ in $v$. Let $S$ be a set. We write $\ku S$ for the free vector space on $S$. Let $A$ be an algebra. By an $A$-module, we mean a left $A$-module. If $S$ is a subset of an $A$-module $M$ and $B\subseteq A$, then $BS$ denotes the set of all $bs$ with $b\in B$ and $s\in S$.

Let $H$ be a finite-dimensional Hopf algebra. We denote by $\Delta$, $\cS$ and $\varepsilon$ the comultiplication, the antipode and the counit of $H$. We will use the Sweedler notation $\Delta(h)=h\_{1}\ot h\_{2}$ for the comultiplication of any $h\in H$, and for the coaction $\delta(m)=m\_{-1} \ot m\_{0}$ of an element $m$ belonging to an $H$-comodule.

Recall that $\ydh$ denotes the category of Yetter-Drinfeld modules over $H$, whose objects are the $H$-modules and $H$-comodules $M$ such that for every $h\in H$ and $m\in M$ it holds that
$$
(hm)\_{-1} \ot(hm)\_{0}=h\_{1}m\_{-1}\cS(h\_{3})\ot h\_{2}m\_{0}.
$$

\subsection{}\label{subsub:DH} We consider the Drinfeld double $\D(H)$ of $H$ according to \cite[Theorem 7.1.1]{majid-q}. Namely, $\D(H)$ is $H\ot H^*$ as coalgebra. Meanwhile, the multiplication and the antipode are given by
\begin{align}\label{eq:DH}
\begin{split}
(h\ot f)(h'\ot f')=&\langle f\_{1}, h'\_{1}\rangle\langle f\_{3},\cS_H(h'\_{3})\rangle (hh'\_{2}\ot f'f\_{2}),\\
\noalign{\smallskip}
\cS(h\ot f)=&(1\ot\cS_{H^*}^{-1}(f))(\cS_H(h)\ot\e),\quad\mbox{for every $h,h'\in H$ and $f,f'\in H^*$}.
\end{split}
\end{align}
In consequence, we have that $H$ and $H^{*op}$ are Hopf subalgebras of $\D(H)$.

Recall that the category $\ydh$ is braided equivalent to the category ${}_{\D(H)}\cM$ of $\D(H)$-modules. Namely, if $M\in\ydh$, then $M$ is a $\D(H)$-module by setting
\begin{align}\label{eq:ydh a dh}
(hf)\cdot m=\langle f,m\_{-1}\rangle hm\_{0}
\end{align}
for every $h\in H$, $f\in H^*$ and $m\in M$.

\subsection{}\label{subsub:BV}
The Nichols algebra of $V\in\ydh$ is constructed as follows, see for instance \cite[$\S$2.1]{AS2}.  First, we consider the tensor algebra $T(V)$ as a graded braided Hopf algebra in $\ydh$ by defining
$$
\Delta(v)=v\ot1+1\ot v,\quad\cS(v)=-v\quad\mbox{and}\quad\varepsilon(v)=0
$$
for all $v\in V$. Let $\cJ(V)$ be the maximal ideal and coideal of $T(V)$ generated by homogeneous elements of degree $\geq2$. Then the Nichols algebra of $V$ is the quotient
$$\BV(V)=T(V)/\cJ(V)$$
which is a graded braided Hopf algebra in $\ydh$. Its homogeneous component of degree $n\in\N$ will be denoted $\BV^n(V)$. Note that $\BV^1(V)=V$ and $\BV^0(V)=\ku$. Moreover, if $\BV(V)$ is finite-dimensional, then its homogeneous component of maximum degree is one-dimensional and it is the space of left and right integrals, see for instance \cite[$\S$2.3 and $\S$3.2]{AG1}.

\subsection{}\label{subsub:boson} The bosonization $\BV(V)\#H$ \cite{radford,majid} is the Hopf algebra structure defined on $\BV(V)\ot H$ in such a way that $H$ is a Hopf subalgebra, $\BV(V)$ is a subalgebra,
\begin{align}\label{eq:boson}
hv=(h\_{1}\cdot v)\#h\_{2}\quad\mbox{and}\quad\Delta(v)=v\ot1+v\_{-1}\ot v\_{0}\quad\mbox{for all $h\in H$ and $v\in V$.}
\end{align}
It is a graded Hopf algebra where its homogeneous component of degree $n\in\N$ is $\BV^n(V)\#H$.

\subsection{}\label{subsub:G} Let $G$ be a finite group. The unity element of $G$ is denoted by $e$. We set $\ku^G=(\ku G)^*$, the dual Hopf algebra of the group algebra $\ku G$. Let $\{\delta_g\}_{g\in G}$ be the dual basis of the canonical basis $\{g\}_{g\in G}$ of $\ku G$. The comultiplication of an element $\delta_g$ is
\begin{align*}
\Delta(\delta_g)=\sum_{t\in G}\delta_t\ot\delta_{t^{-1}g}.
\end{align*}
Let $M$ be a $\ku^G$-module and $g\in G$. Then $M$ is $G$-graded with homogeneous component of degree $g$:
$$
M[g]=\delta_g M=\left\{m\in M\mid f\cdot m=f(g)m\quad\forall f\in\ku^G\right\}.
$$
If $S\subseteq M$, we set $S[g]=S\cap M[g]$. We denote by $\supp M$ the subset of $G$ formed by those elements whose homogeneous component in $M$ is non-zero. The one-dimensional $\ku^G$-module of degree $g$ will be denoted $\ku_g$. If $\ku^G$ is a subalgebra of $A$, then we will consider $A$ as a $\ku^G$-algebra with the adjoint action, that is $fa=\ad(f)a=f\_{1}a\cS(f\_{2})$ for any $a\in A$ and $f\in\ku^G$.

\subsection{The Drinfeld double of a group algebra}\label{subsec:DG}
We denote by $\D(G)$ the Drinfeld Double of $\ku G$. Since $\ku^G$ is a commutative algebra, $\ku^G$ and $\ku G$ are Hopf subalgebras of $\D(G)$. Then the algebra structure of $\D(G)$ is completely determined by
\begin{align*}
&&&&\delta_h\,g=g\,\delta_{g^{-1}hg}&&\quad\forall g,h\in G,\quad\mbox{cf. \eqref{eq:DH}.}
\end{align*}

We will define Verma modules in $\S$\ref{subsec:verma} by inducing from the simple $\D(G)$-modules. These are well-known because they are equivalent to the simple objects in $\ydg$ and a description of these last can be found for instance in \cite[Subsection 3.1]{AG1}. We recall this description but in the context of modules over $\D(G)$.

Let $\cO_g$ be the conjugacy class of $g\in G$, $\cC_g$ the centralizer of $g$ and $(U,\varrho)$ an irreducible representation of $\cC_g$. The $\ku G$-module induced by $(U,\varrho)$,
\begin{align}\label{eq:Mgrho}
M(g,\varrho)=\Ind_{\cC_g}^GU=\ku G\ot_{\ku\cC_g}U,
\end{align}
is also a $\ku^G$-module if we define the action by
$$f\cdot (x\,\ot_{\ku\cC_g} u)=\langle f,xgx^{-1}\rangle x\,\ot_{\ku\cC_g} u,\quad\mbox{for all $f\in\ku^G$, $x\in G$ and $u\in U$}.$$
Then $x\,\ot_{\ku\cC_g} u$ is of $G$-degree $xgx^{-1}$ and $\supp M(g,\varrho)=\cO_g$. Note that $\dim M(g,\varrho)=\#\cO_g\cdot\dim U$.

Therefore $M(g,\varrho)$ is a $\D(G)$-module. Moreover, $M(g,\varrho)$ is simple and every simple $\D(G)$-module  is of this form by \cite[Proposition 3.1.2]{AG1}.

\begin{definition}
A $\D(G)$-module is {\it of weight $(g,\varrho)$} if it is isomorphic to $M(g,\varrho)$.
\end{definition}

Let $\Sn_3$ be the group of bijections on $\{1,2,3\}$. We set $\sigma=(12)$ and $\tau=(123)$. These two cycles generate $\Sn_3$ and satisfy the relations $\sigma^2=e=\tau^3$ and $\sigma\tau\sigma=\tau^{-1}$. The conjugacy classes of $\Sn_3$ are
$$
\cO_e=\{e\},\,\cO_\sigma=\left\{(12),(13),(23)\right\}\,\mbox{ and }\,\cO_\tau=\left\{(123),(132)\right\}.
$$

Next, we describe the simple $\DSn$-modules which we will consider in $\S$\ref{sec:S3}.

\subsubsection{Simple modules attached to \texorpdfstring{$\sigma$}{sigma}}\label{subsub:s3 sigma simples}

The centralizer $\cC_\sigma$ is just the cyclic subgroup generated by $\sigma$. Then $\cC_\sigma$ has only two irreducible representations: the trivial one and the induced by the sign map $\sgn:\Sn_3\rightarrow\{\pm1\}$. Therefore the simple $\DSn$-modules attached to $\sigma$ are
$$
M(\sigma,+):=M(\sigma,\varepsilon)\quad\mbox{and}\quad M(\sigma,-):=M(\sigma,\sgn).
$$

Let us consider the set of symbols $\left\{\mm{12}_{\pm}, \mm{23}_{\pm}, \mm{13}_{\pm}\right\}$ as a basis of $M(\sigma,\pm)$. Sometimes we write $\mm{\sigma\tau^t}_{\pm}$ instead of $\mm{ij}_{\pm}$, if $\sigma\tau^t=(ij)$, and omit the subscript if there is no place for confusion. Hence the action of $\DSn$ on $M(\sigma,\pm)$ is defined in such a way that $\mm{\sigma\tau^t}_{\pm}$ has $\Sn_3$-degree $\sigma\tau^t$ and
$$
\sigma\cdot \mm{\sigma\tau^t}_{\pm}=\pm\, \mm{\sigma\tau^{-t}}_{\pm}\quad\mbox{and}\quad\tau\cdot \mm{\sigma\tau^t}_{\pm}=\mm{\sigma\tau^{t+1}}_{\pm}.
$$

\subsubsection{Simple modules attached to \texorpdfstring{$\tau$}{tau}}\label{subsub:s3 tau simples}

From now on, we fix a root of the unit $\zeta$ of order $3$. The centralizer $\cC_\tau$ is the cyclic subgroup generated by $\tau$. Then $\cC_\tau$ has (up to isomorphisms) three irreducible representations. These are given by the group maps $\rho_\ell:\cC_{\tau}=\langle\tau\rangle\longmapsto\ku^*$, $\tau\mapsto\zeta^\ell$ for $\ell=0,1,2$. Therefore the simple $\DSn$-modules attached to $\tau$ are
$$
M(\tau,\ell):=M(\tau,\rho_\ell)\quad\mbox{for $\ell=0,1,2$}.
$$

Let us consider the set of symbols $\{\mm{123}_{\ell}, \mm{132}_{\ell}\}$ as a basis of $M(\tau,\ell)$. Sometimes we write $\mm{\tau^{t}}_{\ell}$ instead of $\mm{ijk}_{\ell}$, if $\tau^{t}=(ijk)$, and omit the subscript if there is no place for confusion. Hence the action of $\DSn$ on $M(\tau,\ell)$ is defined in such a way that $\mm{\tau^{\pm1}}_\ell$ is of $\Sn_3$-degree $\tau^{\pm1}$ and
$$
\sigma\tau^{t}\cdot\mm{\tau^{\pm1}}_\ell=\zeta^{\pm t\ell}\mm{\tau^{\mp1}}_\ell\quad\mbox{and}\quad\tau^t\cdot\mm{\tau^{\pm1}}_\ell=\zeta^{\pm t\ell}\,\mm{\tau^{\pm1}}_\ell
$$
for $t=0,1,2$. It is not difficult to check that
\begin{align}\label{eq:iso tau objects as s3 mod}
M(\tau,1)&\longrightarrow M(\tau,2),\quad\mm{\tau^{\pm1}}_1\longmapsto\mm{\tau^{\mp1}}_2
\end{align}
is an isomorphism of $\ku\Sn_3$-modules.

\subsubsection{Simple modules attached to \texorpdfstring{$e$}{e}} \label{subsub:s3 e simples}
Let $\rho:\Sn_3\rightarrow\GL_2(\ku)$ be the map defining the two-dimensional Specht $\Sn_3$-module. Then $(\ku,\varepsilon)$, $(\ku,\sgn)$ and $(\ku^2,\rho)$ is a complete list of non-isomorphic irreducible $\Sn_3$-modules. Therefore the simple $\DSn$-modules attached to $e$ are
$$
M(e,+)=M(e,\e),\quad M(e,-)=M(e,\sgn)\quad\mbox{and}\quad M(e,\rho).
$$
These are concentrated in $\Sn_3$-degree $e$. The modules $M(e,\pm)$ are one-dimensional, we denote by $\mm{e}_{\pm}$ its generators and omit the subscript if there is no place for confusion.

\smallbreak

We can describe the $\ku\Sn_3$-action on $M(e,\rho)$ using the canonical representation of $\ku\Sn_3$ on the vector space spanned by $\{{\bf 1}, {\bf 2}, {\bf 3}\}$. In fact, $\ku\{{\bf 1}, {\bf 2}, {\bf 3}\}$ decomposes into the direct sum $(\ku,\varepsilon)\oplus(\ku^2,\rho)$ where the submodules of weight $(\ku,\varepsilon)$ and $(\ku^2,\rho)$ are spanned by $\left\{{\bf 1}+{\bf 2}+{\bf 3}\right\}$ and $\left\{({\bf 1}-{\bf 2}), ({\bf 2}-{\bf 3})\right\}$, respectively.

Another special basis of $M(e,\rho)$ is the set of symbol $\{\mm{\tau}_{\rho}, \mm{\tau^{-1}}_{\rho}\}$ where
$$\mm{\tau^{\pm1}}_{\rho}=\zeta^{\mp1}{\bf 1}+\zeta^{\pm1}{\bf 2}+{\bf 3}$$
This basis is special because it gives the following isomorphisms of $\ku\Sn_3$-modules
\begin{align}\label{eq:iso of 2 dim simple objects}
M(e,\rho)&\longrightarrow M(\tau,1),\quad\mm{\tau^{\pm1}}_\rho\longmapsto\mm{\tau^{\pm1}}_1.
\end{align}
We omit the subscript in $\mm{\tau^{\pm1}}_{\rho}$ if there is no place for confusion.

\subsubsection{Fusion rules}\label{subsub:fusion rules} Let $W$ and $N$ be simple $\DSn$-modules. We want to decompose the tensor products $W\ot N$ into a direct sum of simple $\DSn$-modules. First, we have a decomposition into the direct sum of two submodules which are not necessarily simple:
\begin{align}\label{eq:first decomposition}
W\ot N=\left(\bigoplus_{g\in\Sn_3}W[g]\ot N[g]\right)\oplus\left(\bigoplus_{g,h\in\Sn_3,\,g\neq h}W[g]\ot N[h]\right).
\end{align}
Note that the first submodule is zero if $\supp W\neq\supp N$.

This decomposition is useful for us because each submodule has a basis which is a transitive $\Sn_3$-set in the sense of the next lemma. Let $\bB_W$ and $\bB_N$ be the bases of $W$ and $N$ given in $\S$\ref{subsub:s3 sigma simples}, \ref{subsub:s3 tau simples} and \ref{subsub:s3 e simples}. Then the sets
\begin{align}\label{eq:basis of first decomposition}
\bB_1=\bigcup_{g\in\Sn_3}\bB_W[g]\ot \bB_N[g]\quad\mbox{and}\quad\bB_2=\bigcup_{g,h\in\Sn_3,\, g\neq h}\bB_W[g]\ot\bB_N[h]
\end{align}
are bases of the first submodule and the second one in \eqref{eq:first decomposition}, respectively.

\begin{lema}\label{le:transitive action}
If $\alpha,\beta\in\bB_\ell$, $\ell=1,2$, then there is $\pi\in\Sn_3$ such that $\pi\cdot\alpha=\lambda\beta$ for some non-zero scalar $\lambda$.
\end{lema}

\begin{proof}
The sets $\bB_W[g]$ and $\bB_N[g]$, $g\in\Sn_3$, are either empty or have only one element $\mm{g}$ except to $M(e,\rho)$, but in this case the basis is $\{\mm{\tau}_{\rho},\mm{\tau^{-1}}_{\rho}\}$. In these bases, we see from the definition that $\pi\mm{g}=\lambda\mm{\pi g\pi^{-1}}$ for some non-zero scalar $\lambda$. We conclude by remarking that $\Sn_3$ acts transitively by conjugation on the sets $\{g\times g\mid g\in\supp W\cap\supp N\}$ and $\{g\times h\mid g\in\supp W,\, h\in\supp N,\, g\neq h\}$.
\end{proof}

As a consequence of the above lemma we have the next remark wich will be useful in $\S$\ref{sec:S3} where the action of $V$, or $\oV$, on $N$ will play the role of $\mu$.

\begin{Rem}\label{obs:description submodule gen by a submodule 1}
Let $\mu:W\ot N\rightarrow N'$ be a map of $\DSn$-modules. Assume there is $\alpha\in\bB_1$, respectively $\alpha\in\bB_2$, such that $\mu(\alpha)=0$. Hence $\mu$ restricted to $\bigoplus_{g\in\Sn_3}W[g]\ot N[g]$, respectively $\bigoplus_{g,h\in\Sn_3,\,g\neq h}W[g]\ot N[h]$, is zero since $\Sn_3$ acts transitively on the basis $\bB_1$, respectively $\bB_2$.
\end{Rem}

\bigbreak

We next list the precise fusion rules only for those tensor products which will appear in Section \ref{sec:S3}. We give the assignments (or describe the submodules) which realize the listed isomorphisms but we leave to the reader the verification that these really are maps of $\DSn$-modules (or $\DSn$-submodules).
\begin{itemize}
\item $M(e,-)\ot M(e,-)\simeq M(e,+)$ and $M(\sigma,\pm)\ot M(e,-)\simeq M(\sigma,\mp)$.
\end{itemize}
The isomorphisms are given by $m\ot\mm{e}\longmapsto m$.
\smallbreak

\begin{itemize}
\item $M(e,\rho)\ot M(e,-)\simeq M(e,\rho)$ and
\item $M(\tau,\ell)\ot M(e,-)\simeq M(\tau,\ell)$ for all $\ell=0,1,2$.
\end{itemize}
The assignments $\mm{\tau^{\pm1}}\ot\mm{e}\longmapsto\pm\, \mm{\tau^{\pm1}}$ give these isomorphisms.

\smallbreak

In the sequel, by abuse of notation, $i\ell$ and $\ell+i$ denote the multiplication and sum module $3$.
\begin{itemize}
\item $M(\tau,\ell)\ot\,M(\tau,\ell)\simeq M(e,+)\oplus M(e,-)\oplus M(\tau,2\ell)$ for all $\ell=0,1,2$.
\end{itemize}
We obtain this isomorphism keeping in mind that
$$
M(e,\pm)\simeq\ku\left\{\mm{\tau}_\ell\ot\mm{\tau^{-1}}_\ell\pm\mm{\tau^{-1}}_\ell\ot\mm{\tau}_\ell\right\}\,\mbox{and}\, M(\tau,2\ell)\simeq\ku\left\{\mm{\tau}_\ell\ot\mm{\tau}_\ell,\mm{\tau^{-1}}_\ell\ot\mm{\tau^{-1}}_\ell\right\}.
$$

\smallbreak
\begin{itemize}
\item $M(\tau,\ell)\ot M(e,\rho)\simeq M(\tau,\ell+1)\oplus M(\tau,\ell+2)$ for all $\ell=0,1,2$.
\end{itemize}
The isomorphism follows by considering the submodules
\begin{align}\label{eq:tau en tau rho}
\left\{\mm{\tau^{\pm1}}_\ell\ot \mm{\tau^{\pm1}}_\rho\right\}\quad\mbox{and}\quad\left\{\mm{\tau^{\pm1}}_\ell\ot \mm{\tau^{\mp1}}_\rho\right\}.
\end{align}

\smallbreak
\begin{itemize}
\item $M(\tau,i)\ot\,M(\tau,j)\simeq M(e,\rho)\oplus M(\tau,k)$ with $\{i,j,k\}=\{0,1,2\}$.
\end{itemize}
Here we use that
\begin{align}\label{eq:tau tensor tau}
\begin{split}
M(e,\rho)\simeq\ku\left\{\mm{\tau}_i\ot \mm{\tau^{-1}}_j,\mm{\tau^{-1}}_i\ot\mm{\tau}_j\right\}\quad\mbox{and}\\
\quad
M(\tau,k)\simeq\ku\left\{\mm{\tau}_i\ot\mm{\tau}_j,\mm{\tau^{-1}}_i\ot\mm{\tau^{-1}}_j\right\}.
\end{split}
\end{align}

\smallbreak
\begin{itemize}
\item $M(\tau,\ell)\ot M(\sigma,-)\simeq M(\sigma,+)\oplus M(\sigma,-)\simeq M(\sigma,-)\ot M(\tau,\ell)$ for all $\ell=0,1,2$.
\end{itemize}
In the first isomorphism $\mm{\sigma\tau^{i}}_{\pm}\in M(\sigma,\pm)$ identifies with the element
\begin{align}\label{eq:tau tensor sigma}
\zeta^{i\ell}\mm{\tau}_\ell\ot\mm{\sigma\tau^{i+1}}_-\mp\zeta^{-i\ell}\mm{\tau^{-1}}_\ell\ot\mm{\sigma\tau^{i+2}}_-
\quad\mbox{for }i=0,1,2,
\end{align}
meanwhile in the second isomorphism, $\mm{\sigma\tau^{i}}_{\pm}$, for $i=0,1,2$, identifies with
\begin{align}\label{eq:sigma tensor tau}
\zeta^{i\ell}\mm{\sigma\tau^{i+2}}_-\ot\mm{\tau}_\ell\mp\zeta^{-i\ell}\mm{\sigma\tau^{i+1}}_-\ot\mm{\tau^{-1}}_\ell
=\zeta^{i\ell}(1\pm\sigma\tau^i)\mm{\sigma\tau^{i+2}}_-\ot \mm{\tau}_\ell
\end{align}

\smallbreak
\begin{itemize}
\item $M(\tau,\ell)\ot M(\sigma,+)\simeq M(\sigma,+)\oplus M(\sigma,-)$ for all $\ell=0,1,2$.
\end{itemize}
Here we take the assignments $\zeta^{i\ell}\mm{\tau}_\ell\ot\mm{\sigma\tau^{i+1}}_+\pm\zeta^{i\ell}
\mm{\tau^{-1}}_\ell\ot\mm{\sigma\tau^{i+2}}_+\longmapsto\mm{\sigma\tau^{i}}_{\pm}$ for $i=0,1,2$.

\smallbreak
\begin{itemize}
\item $M(\sigma,-)\ot M(\sigma,\pm)\simeq M(e,\mp)\oplus M(e,\rho)\oplus\bigoplus_{\ell=0, 1, 2} M(\tau,\ell)$.
\end{itemize}
For this isomorphism we use that $\ku\left\{\sum_{i=0}^2\mm{\sigma\tau^i}_-\ot\mm{\sigma\tau^i}_\pm\right\}$ is a one-dimensional submodule; the maps
\begin{align}\label{eq:rho en sigma sigma}
\begin{split}
M(e,\rho)\longrightarrow M(\sigma,-)\ot M(\sigma,-),\quad \mm{\tau^j}_{\rho}\longmapsto\sum_{i=0}^2\zeta^{-ij}\mm{\sigma\tau^i}_-\ot\mm{\sigma\tau^i}_-\quad\mbox{and}
\\
M(e,\rho)\longrightarrow M(\sigma,-)\ot M(\sigma,+),\quad \mm{\tau^j}_{\rho}\longmapsto\sum_{i=0}^2j\zeta^{-ij}\mm{\sigma\tau^i}_-\ot\mm{\sigma\tau^i}_+,
\end{split}
\end{align}
with $j=\pm1$, define inclusions of $\DSn$-modules and $M(\tau,\ell)$ is included as $\DSn$-module by identifying the element $\mm{\tau^i}_\ell$ of $M(\tau,\ell)$ with the element
\begin{align}\label{eq:tau en sigma sigma}
\begin{split}
(\zeta^\ell+\zeta^{-\ell}\tau^{-i}+\tau^i)\,\mm{\sigma}_-\ot\mm{\sigma\tau^i}_-\,\in\, M(\sigma,-)\ot M(\sigma,-),\quad\mbox{respectively}\\
i(\zeta^\ell+\zeta^{-\ell}\tau^{-i}+\tau^i)\,\mm{\sigma}_-\ot\mm{\sigma\tau^i}_+\,\in\, M(\sigma,-)\ot M(\sigma,+),
\end{split}
\end{align}
for $i=\pm1$ and $\ell=0,1,2$.

\smallbreak
\begin{itemize}
\item $M(\sigma,-)\ot M(e,\rho)\simeq M(\sigma,+)\oplus M(\sigma,-)$.
\end{itemize}
Here we have to identify $\mm{\sigma\tau^i}_{\pm}\in M(\sigma,\pm)$, for $i=1,2,3$, with the element
\begin{align}\label{eq:sigma tensor rho}
\zeta^{i}\mm{\sigma\tau^{i}}_-\ot\mm{\tau}_\rho\mp\zeta^{-i}\mm{\sigma\tau^{i}}_-\ot \mm{\tau^{-1}}_\rho
=\zeta^{i}(1\pm\sigma\tau^i)\mm{\sigma\tau^{i}}_-\ot \mm{\tau}_\rho.
\end{align}

\section{A quantum group at a non-abelian group}\label{sec:QnonabG}

Through this section, we fix a finite non-abelian group $G$ and a Yetter-Drinfeld module $V\in\ydg$ such that its Nichols algebra $\BV(V)$ is finite-dimensional. We denote by $\D$ the Drinfeld double of the bosonization $\BV(V)\#\ku G$. For shortness we say that $\D$ is a {\it quantum group at a non-abelian group}.

In the first part of the section we describe the algebra structure of $\D$. Then we introduce and study the Verma modules for $\D$.

\begin{definition}\label{def:oV}
We set $\oV$ to be $V^*$ endowed with the Yetter-Drinfeld module structure over $\ku^G$ defined by the following properties:
\begin{align}\label{eq:oV}
\langle f\cdot y,x\rangle=\langle f,\cS(x\_{-1})\rangle\langle y,x\_{0}\rangle\quad\mbox{and}\quad
\langle y,g\cdot x\rangle=\langle y\_{-1},g\rangle\langle y\_{0},x\rangle
\end{align}
for every $y\in\oV$, $x\in V$, $g\in G$ and $f\in\ku^G$.
\end{definition}

It is a straightforward computation to check that these structures satisfy the compatibility for Yetter-Drinfeld modules. Also, this is a consequence of the next lemma. Recall the Hopf algebra structure of a bosonization in $\S$\ref{subsub:boson}.

\begin{lema}\label{le:iso dual op}
The algebra map $\varphi:\BV(\overline{V})\#\ku^G\longrightarrow(\BV(V)\#\ku G)^{*op}$
defined by
\begin{align*}
\langle\varphi(f),g\rangle&=\langle f,g\rangle\quad\mbox{and}\quad\langle\varphi(f),\BV^n(V)\#\ku G\rangle=0,\\
\langle\varphi(y),x\#g\rangle&=\langle y,x\rangle\quad\mbox{and}\quad\langle\varphi(y),\BV^m(V)\#\ku G\rangle=0
\end{align*}
for all $g\in G$, $f\in\ku^G$, $x\in V$, $y\in\oV$, $n>0$ and $m\neq1$, is an isomorphism of graded Hopf algebras.

In particular, the Hilbert series of the Nichols algebras $\BV(V)$ and $\BV(\oV)$ are equals.
\end{lema}

\begin{proof}
We can deduce that $(\BV(V)\#\ku G)^{*op}\simeq R\#\ku^G$ where $R$ is the Nichols algebra of its homogeneous space of degree $1$ following for instance \cite[Section 2]{B}. Also, we see from the definition in the statement that $\varphi:\oV\#\ku^G\longrightarrow(\BV(V)\#\ku G)^{*op}$, with $\varphi(y\#f)=\varphi(f)\varphi(y)$ for all $y\in\oV$ and $f\in\ku^G$, is a linear map which is bijective in degree $0$ and $1$. Therefore the lemma follows if we show that
\begin{align}\label{eq:coradical del dual}
\Delta\varphi(\delta_g)&=\sum_{t\in G}\varphi(\delta_t)\ot\varphi(\delta_{t^{-1}g}),\quad
\varphi(\delta_g\delta_h)=\varphi(\delta_h)\varphi(\delta_g),\\
\label{eq:inf braiding del dual}
\Delta\varphi(y)&=\varphi(y)\ot1+\varphi(y\_{-1})\ot\varphi(y\_{0})\quad\mbox{and}\quad
\varphi(y)\varphi(\delta_g)=\sum_{t\in G}\varphi(\delta_{t^{-1}g})\varphi(\delta_{t}\cdot y)
\end{align}
for all $g,h\in G$ and $y\in\oV$. In fact, \eqref{eq:coradical del dual} ensures that $\varphi_{|\ku^G}$ is a Hopf algebra map. By \eqref{eq:inf braiding del dual}, we determine that $\oV$ is the space of coinvariants in degree $1$ of $(\BV(V)\#\ku G)^{*op}$ with respect to the projection over $\ku^G$ and the corresponding Yetter-Drinfeld structure is given by Definition \ref{def:oV}. Hence $R$ is the Nichols algebra of $\oV$.

We check the first equality of \eqref{eq:inf braiding del dual}, the remainder ones can be checked in a similar way. As $\BV(V)\#\ku G$ is a graded Hopf algebra it is enough to see that
\begin{align*}
\langle\Delta\varphi(y),a\ot(x\#b)\rangle=&\langle\varphi(y),a(x\#b)\rangle=\langle\varphi(y),(a\cdot x)\#ab\rangle=\langle y,a\cdot x\rangle=\langle y\_{-1},a\rangle\langle y\_{0},x\rangle=\langle \varphi(y\_{-1}),a\rangle\langle\varphi(y\_{0}),x\rangle\\
=&\langle \varphi(y),a\rangle\langle 1,x\#b\rangle+\langle \varphi(y\_{-1}),a\rangle\langle\varphi(y\_{0}),x\#b\rangle=\langle \varphi(y)\ot1+\varphi(y\_{-1})\ot\varphi(y\_{0}),a\ot(x\#b)\rangle
\end{align*}
for all $a,b\in G$ and $x\in V$.
\end{proof}

\begin{convention}
We identify the Hopf subalgebra $(\BV(V)\#\ku G)^{*op}$ of $\D$ with $\BV(\oV)\#\ku^G$ by invoking the above lemma.
\end{convention}

\begin{lema}
The quantum group $\D$ at a non-abelian group can be presented  as an algebra generated by the elements belonging to $V$, $\oV$, $\ku G$ and $\ku^G$ subject to their relations in $\BV(V)$, $\BV(\oV)$, $\ku G$ and $\ku^G$, plus the commutation rules
\begin{align}
\label{eq:reglas DH boson}g\,x=&(g\cdot x)\,g,&\delta_g\,y=&\sum_{t\in G}(\delta_t\cdot y)\delta_{t^{-1}g},\\
\label{eq:reglas DH boson cruzadas}\delta_g\,x=&\sum_{t\in G}\langle\delta_t,x\_{-1}\rangle x\_{0}\delta_{t^{-1}g},
&y\,g=&\langle y\_{-1},g\rangle\, gy\_{0},
\end{align}
\begin{align}
\label{eq:reglas DH xy}
y\,x-\langle y\_{-1},x\_{-1}\rangle\,x\_{0}y\_{0}&=\langle y,x\rangle1+\langle y\_{-2},x\_{-2}\rangle
\langle y\_{0},\cS(x\_{0})\rangle\,x\_{-1}y\_{-1},\\
\label{eq:double of G}\delta_h\,g&=g\,\delta_{g^{-1}hg}.
\end{align}
for all $g,h\in G$, $x\in V$ and $y\in\oV$.
\end{lema}

\begin{proof}
The equations \eqref{eq:reglas DH boson} correspond to the bosonization, see \eqref{eq:boson}. Meanwhile \eqref{eq:reglas DH boson cruzadas}, \eqref{eq:reglas DH xy} and \eqref{eq:double of G} follow from \eqref{eq:DH}.
\end{proof}

\begin{lema}
The subalgebra of $\D$ generated by $\ku G$ and $\ku^G$ is a Hopf subalgebra isomorphic to $\D(G)$ and it is the coradical of $\D$. In particular, $\D$ is non-pointed.
\end{lema}
\begin{proof}
It follows from Lemma \ref{le:iso dual op} and \eqref{eq:double of G}.
\end{proof}

Due to the above lemmata, a quantum group at a non-abelian group has a triangular decomposition, that is
\begin{align}\label{eq:D decomposition}
\D=\BV(V)\ot\D(G)\ot\BV(\oV),
\end{align}
and it is a $\Z$-graded algebra by setting
\begin{align}\label{eq:Zgraded}
\deg V=-1,\quad\deg\D(G)=0,\quad\deg\oV=1.
\end{align}

\medbreak

\begin{convention}
We consider $V$ and $\oV$ as Yetter-Drinfeld modules over $\D(G)$ with the adjoint action and the same coaction as $\ku G$-comodule and  $\ku^G$-comodule, respectively.

That is possible because the rules \eqref{eq:reglas DH boson} and \eqref{eq:reglas DH boson cruzadas} guarantee that $V$ and $\oV$ are stable by the adjoint action of $\D(G)$, {\it  i.~e.} $\ad(h) x=h\_{1}x\cS(h\_{2})\in V$ and $\ad(h) y=h\_{1}y\cS(h\_{2})\in\oV$ for all $h\in\D(G)$, $x\in V$ and $y\in\oV$. Also, $\D(G)=\ku G\ot\ku^G$ as coalgebra.
\end{convention}

\smallbreak

We extend these structures to $\BV(V)$ and $\BV(\oV)$. Hence the bosonization
\begin{align}\label{eq:D mas y D menos}
\D^{\leq0}=\BV(V)\#\D(G),\quad\mbox{respectively}\quad\D^{\geq0}=\BV(\oV)\#\D(G),
\end{align}
identifies with the subalgebra of $\D$ generated by $\D(G)$ and $V$, respectively $\oV$. 

\begin{Rem}\label{obs:V as DGmod}
The adjoint action of $\D(G)$ on $V$ coincides with the action defined by the equivalence of categories between $\ydg$ and ${}_{\D(G)}\cM$, see \eqref{eq:ydh a dh}.
\end{Rem}

We would like to remark other facts about $V$ and $\oV$. We refer to \cite{AG1} for details about the item (iv) below.
\begin{lema}\phantomsection\label{le:V oV son duales sobre DG}
\begin{enumerate}\renewcommand{\theenumi}{\roman{enumi}}\renewcommand{\labelenumi}{
(\theenumi)}
\item $\oV$ is the dual object of $V$ in the tensor category ${}_{\D(G)}\cM$.
\item $\BV(V)$ and $\BV(\oV)$ are the Nichols algebras of $V$ and $\oV$ in ${}_{\D(G)}^{\D(G)}\mathcal{YD}$, respectively.
\item $\BV(V)$ and $\BV(\oV)$ are the Nichols algebras of $V$ and $\oV$ in ${}_{\D(G)}\cM$, respectively.
\item $\BV(\oV)$ is isomorphic to the opposite and copposite Hopf algebra $\BV(V)^{*bop}$ in ${}_{\D(G)}\cM$.
\end{enumerate}
\end{lema}

\begin{proof}
(i) Let $y\in\oV$, $x\in V$ and $g\in G$. By \eqref{eq:reglas DH boson cruzadas}, we have that
$\langle\ad(g)y, x\rangle=\langle y\_{-1},g^{-1}\rangle\langle y\_{0},x\rangle$. On the other hand, \eqref{eq:oV} and \eqref{eq:reglas DH boson} imply that
$\langle y, \ad(g^{-1})x\rangle=\langle y, g^{-1}\cdot x\rangle=\langle y\_{-1},g^{-1}\rangle\langle y\_{0},x\rangle$. Then $\langle\ad(g)y, x\rangle=\langle y,\ad\cS(g)x\rangle$. In a similar way, we see that $\langle\ad(\delta_g)y, x\rangle=\langle y,\ad\cS(\delta_g)x\rangle$.

(ii) $\BV(V)$ and $\BV(\oV)$ are braided Hopf algebras in ${}_{\D(G)}^{\D(G)}\mathcal{YD}$, because $\D^{\leq0}$ and $\D^{\geq0}$ are Hopf algebras, which satisfy the defining properties of a Nichols algebra

(iii) follows from (ii) because the braiding of ${}_{\D(G)}^{\D(G)}\mathcal{YD}$ on $V$ coincides with that of ${}_{\D(G)}\cM$ and the same holds for $\oV$.

(iv) Let $\widetilde{V}$ be the dual object of $V$ in $\ydg$. By \cite[Theorem 3.2.30]{AG1}, $\BV(\widetilde{V})\simeq\BV(V)^{*bop}$ in $\ydg$.
We said before that the adjoint action of $\D(G)$ on $V$ coincides with the action defined by the functor given the equivalence of categories between $\ydg$ and ${}_{\D(G)}\cM$. Then, by (i), $\oV$ is the image of $\widetilde{V}$ by this functor and (iv) follows because the braidings of these categories are equal via this functor.
\end{proof}

\subsection{Verma modules}\label{subsec:verma}
A classical technique in Representation Theory is to study modules induced by simple modules of a subalgebra. Such is the case of the Verma modules for quantum groups, see for instance \cite[Chapter 7]{dixmier}, \cite[Chapter 5]{Jantzen} and \cite[Chapter 3]{lusztig}. Following this idea, we shall induce from the subalgebra $\D^{\geq0}$. 

Every simple $\D^{\geq0}$-module is isomorphic to a simple $\D(G)$-module where $\BV(\overline{V})$ acts via the counit. This holds because $\BV(\oV)$ is local and hence $\ker(\varepsilon)\#\D(G)$ is the Jacobson radical of $\D^{\geq0}$.

\begin{definition}\label{def:verma}
Let $M(g,\varrho)$ be a simple $\D(G)$-module. The {\it Verma module} $\fM(g,\varrho)$ is the $\D$-module induced by $M(g,\varrho)$ seen as a module over $\D^{\geq0}$. Explicitly,
$$
\fM(g,\varrho)=\D\ot_{\D^{\geq0}}M(g,\varrho).
$$
\end{definition}

\medbreak

We fix a simple $\D(G)$-module $M(g,\varrho)$ and set $\fM=\fM(g,\varrho)$. Immediately from the definition, we get that $\fM$ is free as $\BV(V)$-module of rank $\dim M(g,\varrho)=\#\cO_g\cdot\dim\varrho$. Moreover,
\begin{align}\label{eq:verma as dg module}
\fM=\BV(V)\ot M(g,\varrho)\quad\mbox{in }\,{}_{\D(G)}\cM
\end{align}
since $h(x\ot m)=(hx)\ot m =ad(h\_{1})xh\_{2}\ot m=ad(h\_{1})x\ot h\_{2}m$, for $h\in\D(G)$, $x\in\BV(V)$ and $m\in M(g,\varrho)$,
and the last term is the definition of the action in the tensor product of two $\D(G)$-modules.

Also $\fM$ inherits the $\Z$-grading of $\D$ and its homogeneous spaces are $\D(G)$-submodules. Namely, its homogeneous space of degree $n\leq0$ is
$$\fM^{n}=\BV^{-n}(V)\ot M(g,\varrho).$$
Thus $\fM$ turns out to be an $\Z$-graded $\D$-module since
\begin{align}\label{eq:verma G y N graded}
V\,\fM^{n}=\fM^{n-1}\quad\mbox{and}\quad \oV\,\fM^{n}\subseteq \fM^{n+1}.
\end{align}
In fact, the first equality holds as the action by an element of $V$ is just the multiplication in the Nichols algebra and this is generated by $V$. We proceed by induction to prove the second inclusion. If $n=0$, then $\fM^0=\ku1\ot M(g,\varrho)$ on which $\oV$ acts by zero. If $n<0$, then
$$
\oV\,\fM^{n-1}=\oV V\,\fM^{n}\subseteq V\oV\,\fM^{n}+\D(G)\fM^{n}\subseteq\fM^{n},
$$
where the first inclusion follows from \eqref{eq:reglas DH xy} and the second one by inductive hypothesis.

As $\D^{\leq0}$-module, $\fM$ is generated by any element of degree 0, that is
\begin{align}\label{eq:verma d mas ciclico}
\fM=\D^{\leq0}(1\ot m)\quad\forall m\in M.
\end{align}
In fact, $\D^{\leq0}(1\ot m)=\BV(V)\ot \D(G) m=\BV(V)\ot M(g,\varrho)$ since $M(g,\varrho)$ is $\D(G)$-simple. However we have more generators for a Verma module as $\D^{\leq0}$-module.

\begin{lema}\label{le:si esta 1 mas algo entonces es todo el verma}
If $\fm\in \fM^0$ is non-zero, then $\fM=\D^{\leq0}(\fm+\fn)$
for any $\fn\in\oplus_{n<0}\fM^n$.
\end{lema}

\begin{proof}
By \eqref{eq:verma d mas ciclico} the lemma follows if $\fn=0$. Otherwise, we write $\fn=\fn_1+\fn_2$ with $0\neq\fn_{1}\in\fM^{n_1}$ and $\fn_2\in\oplus_{n<n_1}\fM^n$. By \eqref{eq:verma d mas ciclico} there is $z\in\D^{\leq0}$ such that $z\fm=\fn_1$, moreover $z\in\BV^{n_1}(V)\#\D(G)$. Then
$$
(\fm+\fn)-z(\fm+\fn)=\fm+(\fn_2-z\fn_1-z\fn_2)\in\D^{\leq0}(\fm+\fn)
$$
and the maximum degree of $(\fn_2-z\fn_1-z\fn_2)$ is smaller than $n_1$. Hence the lemma follows by induction in the maximum degree of $\fn$ since $\BV(V)$ is finite-dimensional and $V\fM^{-n_{top}}=0$ if $n_{top}$ is the maximum degree of $\BV(V)$.
\end{proof}

\smallbreak

Using the above lemma we prove one of the main properties of a Verma module.
\begin{thm}\label{teo:unico submod max}
A Verma module has a unique maximal $\D$-submodule and it is homogeneous.
\end{thm}

\begin{proof}
If $\fN$ is a strict $\D$-submodule of $\fM$, then there exists a non-zero negative integer $n_\fN$ such that $\fN\subseteq\oplus_{n<n_{\fN}}\fM^n$ by Lemma \ref{le:si esta 1 mas algo entonces es todo el verma}. Hence the sum $\fX$ of all strict $\D$-submodules is the unique maximal $\D$-submodule of $\fM$.

Let $\sum_{n}\fn_n\in\fX$ with $\fn_n\in\fM^n$. If we see that $\fn_n\in\fX$ for all $n$, then $\fX$ is homogeneous.

Otherwise, without loss of generality, we can assume $\fn_n\notin\fX$ with $n$ maximal, then $\D\fn_n=\fM$. Thus there is $z\in\BV^n(\oV)\#\D(G)$ such that $0\neq z\fn_n=1\ot m\in\fM^0$ and hence $z\sum_{n}\fn_n=1\ot m+\widetilde\fn$ with $\widetilde\fn\in\oplus_{n<0}\fM^n$. Then $\D\sum_{n}\fn_n=\fM$ by Lemma \ref{le:si esta 1 mas algo entonces es todo el verma} but this is not possible because $\sum_{n}\fn_n\in\fX\subsetneq\fM$.
\end{proof}

As it is common, we introduce the highest-weight modules in such a way that a Verma module is a highest-weight module. The weights in our case are the simple $\D(G)$-modules which can have dimension greater than one.

\begin{definition}\label{def:highest modules}
Let $N$ be a $\D$-module and $M\subset N$ a simple $\D(G)$-submodule of weight $(g,\varrho)$. Assume that $N$ is generated as $\D$-module by $M$.

We say that $N$ is a {\it highest-weight module of weight $(g,\varrho)$} if $\oV M=0$.

We say that $N$ is a {\it lowest-weight module of weight $(g,\varrho)$} if $V M=0$.
\end{definition}

Hence we have that
\begin{align}\label{eq:highest module}
N=\D M=\BV(V)M=\D^{\leq0}m\quad\forall m\in M
\end{align}
if $N$ is a highest-weight module, and
\begin{align}\label{eq:lowest module}
N=\D M=\BV(\oV)M=\D^{\geq0}m\quad\forall m\in M
\end{align}
in case that $N$ is a lowest-weight module. These follow from the decomposition \eqref{eq:D decomposition} of $\D$ and since $M$ is $\D(G)$-simple.

\smallbreak

We set $\fMsoc=\BV^{n_{top}}(V)\ot M(g,\varrho)$ where $n_{top}$ is the maximum degree of the Nichols algebra. Note that $\fMsoc$ is simple as a $\D(G)$-module since $\BV^{n_{top}}(V)$ is one-dimensional.

\begin{thm}\label{teo:socle simple}
The socle of the Verma module $\fM$ is simple as a $\D$-module and equals $\BV(\oV)\fMsoc$.
\end{thm}

\begin{proof}
The socle is simple if we show that $\fMsoc\subset\D\fm$ for any homogeneous element $\fm\neq0$ of degree $-n$ in $\fM$ with $n<n_{top}$. To show that, we write $\fm=\sum_{i}z_i\ot m_i$ with $z_i\in\BV^n(V)$ and $\{m_i\}\subset M(g,\varrho)$ linearly independent. We pick $z_i$. Since $\BV^{n_{top}}(V)$ is the space of integrals of the Nichols algebra, there is $x_1\in V$ such that $0\neq x_1z_i\in\BV^{n+1}(V)$. Then $0\neq x_1\fm\in\D\fm$ is an homogeneous element of degree $-n-1$. Hence $x_{n_{top}-n}\cdots x_{1}\fm\neq0$ for appropriated  $x_{n_{top}-n},\dots, x_{1}\in V$ and therefore $\fMsoc\subset\D\fm$ because $\fMsoc$ is $\D(G)$-simple.

Finally, the socle is a lowest-weight module because it is generated by $\fMsoc$. Therefore the socle is equal to $\BV(\oV)\fMsoc$ by \eqref{eq:lowest module}.
\end{proof}

As a direct consequence we obtain a criterion for the simplicity of a Verma module. Recall that the Hilbert series of $\BV(V)$ and $\BV(\oV)$ are equal by Lemma \ref{le:iso dual op}.

\begin{Cor}\label{cor:verma simple xtop} Let $\xtop\in\BV^{n_{top}}(V)$, $\ytop\in\BV^{n_{top}}(\oV)$ and $m\in M$. The Verma module $\fM$ is simple as $\D$-module if and only if $\ytop(x_{top}\ot m)\neq0$.
\end{Cor}

\begin{proof}
By \eqref{eq:verma G y N graded}, $\ytop(x_{top}\ot m)\in\fM^0$ and hence it generates $\fM$ by \eqref{eq:verma d mas ciclico}. Then the socle, which is simple by the above theorem, is exactly $\fM$.

Assume now that $\fM$ is a simple $\D$-module. In particular, $\fM$ is generated by $\xtop\ot m$ and thus there is an element $z\in\D$ such that $0\neq z(\xtop\ot m)\in\fM^0$. By \eqref{eq:verma G y N graded}, $z\in\D^{n_{top}}=\sum_{j-i=n_{top}}\BV^i(V)\D(G)\BV^{j}(\oV)$. Since $n_{top}$ is the maximum degree of $\BV(V)$ and $\BV(\oV)$, we have that $z\in\D(G)\BV^{n_{top}}(\oV)$. Finally, we can take $z=\ytop$ because $\BV^{n_{top}}(\oV)$ is one-dimensional.
\end{proof}

Due to the above theorems we can introduce the following $\D$-modules.

\begin{definition}\label{def:max simple socle}
Let $M(g,\varrho)$ be a simple $\D(G)$-module. Then
\begin{itemize}
\item $\fX(g,\varrho)$ denotes the maximal $\D$-submodule of $\fM(g,\varrho)$.
\item $\fL(g,\varrho)$ denotes the head of $\fM(g,\varrho)$.
\item $\fS(g,\varrho)$ denotes the socle of $\fM(g,\varrho)$ as $\D$-module.
\end{itemize}
\end{definition}

The following theorem states that the correspondence $M(g,\varrho)\leftrightsquigarrow\fL(g,\varrho)$, between the sets of simple $\D(G)$-modules and simple $\D$-modules, is bijective.

\begin{thm}\phantomsection\label{teo:bi con L}
\begin{enumerate}\renewcommand{\theenumi}{\roman{enumi}}\renewcommand{\labelenumi}{(\theenumi)}
\item Let $M(g,\varrho)$ be a simple $\D(G)$-module. Then $\fL(g,\varrho)$ is the unique simple highest-weight module of weight $(g,\varrho)$.
\item Every simple $\D$-module is isomorphic to $\fL(g,\varrho)$ for a unique simple $\D(G)$-module $M(g,\varrho)$.
\end{enumerate}
\end{thm}

\begin{proof}
(i) $\fL(g,\varrho)$ is a simple $\D$-module by Theorem \ref{teo:unico submod max}. As $\fM(g,\varrho)$ is generated by $M(g,\varrho)$, $\fL(g,\varrho)$ is so. Moreover, $\oV M(g,\varrho)=0$ then $\fL(g,\varrho)$ is a highest-weight module. The uniqueness follows from the fact that a highest-module $L$ of weight $(g,\varrho)$ is a quotient of $\fM(g,\varrho)$ since $M(g,\varrho)$ turns out to be a $\D^{\geq0}$-submodule of $L$. If also $L$ is $\D$-simple, then $L\simeq\fL(g,\varrho)$ since $\fX(g,\varrho)$ is the unique maximal submodule of $\fM(g,\varrho)$.

(ii) Every $\D$-module $L$ has a simple $\D^{\geq0}$-module, say $M(g,\varrho)$. Then we have a morphism $\fM(g,\varrho)\longrightarrow L$ of $\D$-modules and this map is surjective if $L$ is $\D$-simple. Hence $\fL(g,\varrho)\simeq L$. On the other hand, if there exists another $\fL(g',\varrho')$ isomorphic to $L$, then $M(g,\varrho)\simeq M(g',\varrho')$ because they are highest weights of $L$.
\end{proof}

The correspondence $M(g,\varrho)\leftrightsquigarrow\fS(g,\varrho)$ also is bijective by the next theorem. We set $M(\hat g,\hat\varrho)$ to be the simple $\D(G)$-module isomorphic to $\BV^{n_{top}}(V)\ot M(g,\varrho)$.

\begin{thm}\phantomsection\label{teo:bi con socle}
\begin{enumerate}\renewcommand{\theenumi}{\roman{enumi}}\renewcommand{\labelenumi}{(\theenumi)}
\item Let $M(g,\varrho)$ be a simple $\D(G)$-module. Then $\fS(g,\varrho)$ is the unique simple lowest-weight module of weight $M(\hat g,\hat\varrho)$.
\item Every simple $\D$-module is isomorphic to $\fS(g,\varrho)$ for a unique  simple $\D(G)$-module $M(g,\varrho)$.
\end{enumerate}
\end{thm}

\begin{proof}
The socles of the Verma modules are lowest-weight modules by Theorem \ref{teo:socle simple}. Also, the socles of non-isomorphic Verma modules are non-isomorphic because their lowest-weight components are not. Hence the uniqueness in (i) and (ii) follow from the fact that the sets of simple $\D$-modules and simple $\D(G)$-modules are in bijective correspondence by Theorem \ref{teo:bi con L}.
\end{proof}

We denote by $M(g^*,\varrho^*)$ the dual $\D(G)$-module of $M(g,\varrho)$.
\begin{thm}\label{teo:L dual S}
The dual $\D$-module $\bigl(\fS(g,\varrho)\bigr)^*$ is a highest-weight module of weight $M(\hat g^*,\hat\varrho^*)$. Therefore $\bigl(\fS(g,\varrho)\bigr)^*\simeq\fL(\hat g^*,\hat\varrho^*)$ as $\D$-modules.
\end{thm}

\begin{proof}
Let $\fS_i$, $i\geq -n_{top}$, be the homogeneous component of degree $i$ of $\fS(g,\varrho)$. Then $\bigl(\fS(g,\varrho)\bigr)^*=\oplus_i \fS_i^*$ and $\fS_{-n_{top}}^*\simeq M(\hat g^*,\hat\varrho^*)$ as $\D(G)$-modules. Since $\bigl(\fS(g,\varrho)\bigr)^*$ is simple, it is generated by $\fS_{-n_{top}}^*$. Moreover, $\oV\fS_{-n_{top}}^*=0$. In fact, $\langle\oV\fS_{-n_{top}}^*,\fS_i\rangle=\langle\fS_{-n_{top}}^*,\cS(\oV)\fS_i\rangle=0$ because
$\cS(\oV)\fS_i\subseteq\fS_{i+1}$ and $-n_{top}<i+1$ for all $i\geq -n_{top}$. Hence the theorem follow from Theorem \ref{teo:bi con L}.
\end{proof}

\begin{Rems}
\begin{enumerate}\renewcommand{\theenumi}{\roman{enumi}}\renewcommand{\labelenumi}{(\theenumi)}
\item Even though the maximal submodule of $\fM$ is homogeneous, a submodule is not necessarily homogeneous, cf. Lemma \ref{le:T}.

\item The maximal submodule is not necessarily generated by highest-weight submodules, cf. Theorem \ref{teo:el latice de sigma menos}.

\item The head and the socle of a Verma module are not necessarily isomorphic, cf. Corollary \ref{cor: l tau cero s e rho}.

\item There are examples with $M(\hat g,\hat\varrho)\not\simeq M(g,\varrho)$. For instance, let $R$ be the Nichols algebra considered in \cite{PV}. Then $R^{n_{top}}$ is not necessarily trivial as Yetter-Drinfeld module.
\end{enumerate}
\end{Rems}

\subsection{Highest and lowest weight modules}\label{subsec:algorithm}
We fix a $\D$-module $N$ and a $\D(G)$-submodule $M\subset N$. We will explain how we can compute the $\D$-submodule generated by $M$ under the hypothesis that it is either a lowest-weight or highest-weight module.

We denote by $\mu$, and call it {\it action map}, the restriction to $V\ot M$ of the action of $\D$ over $N$. By abuse of notation, we also denote by $\mu$ the restriction to $\overline{V}\ot M$. The key of our idea is the simple observation that
\begin{align}\label{eq:the action is in ydg}
\mbox{the action map $\mu$ is a morphism in the category }\,{}_{\D(G)}\cM.
\end{align}
Indeed, we want to see that $\mu(h(z\ot m))=h\mu(z\ot m)$ for any $z\in V\cup\overline{V}$, $h\in\D(G)$ and $m\in M$. The action on the tensor product is $h(z\ot m)=h\_{1}\cdot z\ot h\_{2}m$. Then we apply the action map and obtain
$(h\_{1}\cdot z)h\_{2}m=\ad(h\_{1})zh\_{2}m=h\_{1}zS(h\_{2})h\_{3}m=h(zm)=h\mu(z\ot m)$.

\smallbreak

By  \eqref{eq:highest module} and \eqref{eq:lowest module}, the $\D$-submodule generated by $M$ is either $\BV(\oV)M$ or $\BV(V)M$. Hence we can compute $\D M$ following the algorithm described in the next remark.

\begin{Rem}\label{obs:description submodule gen by a submodule 2}
Keep the above hypothesis and notation.
\begin{enumerate}
\item[(I)] Decompose the tensor product $\oV\ot M$, or $V\ot M$ depending on the case, into the direct sum $\oplus S_\ell$ of simple $\D(G)$-modules.
\end{enumerate}
This is possible because $\D(G)$ is semisimple. Moreover, its simple modules are well-know, recall $\S$\ref{subsec:DG}.
\begin{enumerate}
\item[(II)] Apply the action map to each simple $\D(G)$-module $S_\ell$.
\end{enumerate}
The restriction of the action map to $S_\ell$ is either zero or an isomorphism by Schur's Lemma. Therefore the image of the action map is isomorphic as $\D(G)$-module to the direct sum of the simple modules $S_\ell$ that are not annulled. Note that $\mu(S_\ell)=0$ if and only if $\mu(w)=0$ for some $w\in S_\ell$.
\begin{enumerate}
\item[(III)] Repeat the process with the $\D(G)$-submodule $\mu(\oV\ot M)$, or $\mu(V\ot M)$ depending on the case, instead of $M$.
\end{enumerate}
We have to repeat the process as many times as the maximum degree of $\BV(\oV)$ or $\BV(V)$.
\begin{enumerate}
\item[(IV)] The $\D$-submodule generated by $M$ is the sum of all $\D(G)$-submodules obtained in the step (II).
\end{enumerate}
\end{Rem}

\section{The quantum group at the symmetric group \texorpdfstring{$\Sn_3$}{S3} attached to the 12-dimensional Fomin-Kirillov algebra}\label{sec:S3}

Throughout this section $V=\ku\{\xij{12},\xij{23},\xij{13}\}$ is the Yetter-Drinfeld module over $\ku\Sn_3$ given by
$$
g\cdot\xij{ij}=\sgn(g)\,x_{g(ij)g^{-1}}\quad\mbox{and}\quad(\xij{ij})\_{-1}\ot(\xij{ij})\_{0}=(ij)\ot\xij{ij}
$$
for any transposition $(ij)$ and $g\in\Sn_3$. Let $\oV\in\ydstresd$ be the Yetter-Drinfeld module attached to $V$ from Definition \ref{def:oV}. From \cite{MS} we know that $\BV(V)$ is isomorphic to the $12$-dimensional Fomin-Kirillov algebra introduced in \cite{FK}.

We denote by $\D$ the Drinfeld double of the bosonization $\BV(V)\#\ku\Sn_3$. The aim of this section is to apply the results of the previous section in the specific example of this quantum group.

\bigbreak

We have to consider $V$ and $\oV$ as $\DSn$-modules with the adjoint action in $\D$. Using Remark \ref{obs:V as DGmod}, we see that $V\simeq M(\sigma,-)$ via the assignment
$$
V\longrightarrow M(\sigma,-),\quad\xij{ij}\longmapsto\mm{ij}
$$
for every transposition $(ij)$.

By Lemma \ref{le:V oV son duales sobre DG}, $\oV$ is the dual object of $V$ in the category ${}_{\DSn}\cM$. We denote by $\{\yij{12},\yij{23},\yij{13}\}$ the basis of $\oV$ dual to $\{\xij{12},\xij{23},\xij{13}\}$, that is $\langle\yij{ij},\xij{lk}\rangle=\delta_{(ij),(lk)}$.  Then it is not difficult to check that $\oV\simeq M(\sigma,-)$ via the assignment
$$
\oV\longrightarrow M(\sigma,-),\quad\yij{ij}\longmapsto\mm{ij}
$$
for every transposition $(ij)$.

\medbreak

The defining relations of the Nichols algebra $\BV(V)$ are
\begin{align}\label{eq:rel of BV}
\notag \xij{12}^2,&\quad \xij{13}^2,\quad \xij{23}^2,\\
\xij{12}\xij{13}+&\xij{23}\xij{12}+\xij{13}\xij{23}\quad\mbox{and}\\
\notag \xij{13}\xij{12}+&\xij{12}\xij{23}+\xij{23}\xij{13}.
\end{align}
We denote by $\B$ the basis of $\BV(V)$ which is obtained by choosing one element per row of the next list and multiply them from top to bottom, see {\it e. ~g.} \cite{grania}:
\begin{align*}
1, &\, \xij{12}, \\
1, &\, \xij{13},\, \xij{13}\xij{12},\\
1, &\, \xij{23}.
\end{align*}
We set $\B^n=\B\cap\BV^n(V)$, $n\geq0$. The element of maximum degree in $\B$ is
$$\xtop=\xij{12}\xij{13}\xij{12}\xij{23}\in\B^4.$$

\begin{lema}\label{descomposicion de BV as DS3mod}
We have the following isomorphisms of $\DSn$-modules.
\begin{enumerate}\renewcommand{\theenumi}{\roman{enumi}}\renewcommand{\labelenumi}{(\theenumi)}
\item\label{item: 0 4} $\BV^0(V)\simeq\BV^4(V)\simeq M(e,+)$,
\item\label{item: 1 3} $\BV^1(V)\simeq\BV^3(V)\simeq M(\sigma,-)$; the last isomorphism is given by the assignment
$$
\mm{12}\longmapsto\xij{13}\xij{12}\xij{23},\quad \mm{23}\longmapsto-\xij{12}\xij{13}\xij{12},\quad \mm{13}\longmapsto\xij{12}\xij{13}\xij{23}.
$$
\item\label{item:descomposion de BV tau}$\BV^2(V)\simeq M(\tau,1)\oplus M(\tau,2)$, the isomorphism is given by
\begin{align*}
\mm{\tau}_\ell&\longmapsto(\zeta^{\ell}-1) x_{(12)}x_{(23)}+(\zeta^{-\ell}-1) x_{(13)}x_{(12)},\\
\mm{\tau^{-1}}_\ell&\longmapsto(\zeta^{\ell}-\zeta^{-\ell})x_{(12)}x_{(13)}+(1-\zeta^{-\ell}) x_{(13)}x_{(23)}\quad\mbox{for $\ell=1,2$}.
\end{align*}
\end{enumerate}
\end{lema}

\begin{proof} (i) For $\BV^0(V)$ the isomorphism is clear. For $\BV^4(V)$ it is enough to see that
\begin{align*}
\sigma\cdot \xtop&=\sgn^4(\sigma)x_{(12)(12)(12)}x_{(12)(13)(12)}x_{(12)(12)(12)}x_{(12)(23)(12)}=x_{(12)}x_{(23)}x_{(12)}x_{(13)}\\
&=-\xij{12}\xij{13}\xij{23}\xij{13}=\xij{12}\xij{13}\xij{12}\xij{23}=\xtop\quad\mbox{by \eqref{eq:rel of BV}.}
\end{align*}

To prove (ii) we note that $\sigma\cdot x_{(13)}x_{(12)}x_{(23)}=-x_{(13)}x_{(12)}x_{(23)}$ and
\begin{align*}
\sigma\cdot x_{(12)}x_{(13)}x_{(23)}=\sgn(\sigma)^3x_{(12)(12)(12)}x_{(12)(13)(12)}x_{(12)(23)(12)}=-x_{(12)}x_{(23)}x_{(13)}=x_{(12)}x_{(13)}x_{(12)}.
\end{align*}

(iii) follows from \eqref{eq:tau en sigma sigma} and using \eqref{eq:rel of BV}.
\end{proof}

\subsection{Description of the action on a Verma module}
We fix a simple $\D$-module $M$ and take the basis $\bB_M$ of $M$ which consists of elements of the form $\mm{g}$, $g\in\Sn_3$, recall $\S$\ref{subsub:s3 sigma simples}, $\S$\ref{subsub:s3 tau simples} and $\S$\ref{subsub:s3 e simples}.

Let $\fM$ be the Verma module of $M$. Since $\D\simeq\BV(V)\ot\DSn\ot\BV(\oV)$, a basis of $\fM$ is the set of elements
$$
x\mm{g}=x\ot\mm{g}\quad\forall x\in\B,\,\mm{g}\in\bB_M.
$$
Then the action of $\BV(V)$ on $\fM$ is given just by the multiplication:
$$
z\cdot x\mm{g}=(zx)\mm{g}\quad\forall z\in\BV(V).
$$
The action of $\DSn$ is the diagonal action by \eqref{eq:verma as dg module}:
$$
h\cdot x\mm{g}=\ad h\_{1}(x)\ot h\_{2}\cdot\mm{g}\quad\forall h\in\DSn.
$$

Computing the action of $\BV(\oV)$ is more laborious. We have to use the commutation rules between the generators of $\BV(\oV)$ and $\BV(V)$ given by \eqref{eq:reglas DH xy}. In our case \eqref{eq:reglas DH xy} is rewritten as follow
\begin{align}\label{eq:de los ij iguales}
\yij{ij}\xij{ij}&=1+(ij)(\delta_{(ij)}-\delta_e)-\xij{ij}\yij{ij}\quad\mbox{and}\\
\label{eq:de los ij no iguales}
\yij{ik}\xij{ij}&=(ij)(\delta_{(ik)}-\delta_{(ik)(ij)})-\xij{ij}\yij{jk}
\end{align}
for all distinct transpositions $(ij)$ and $(ik)$. However, if we know the action of $\yij{12}$ on $\fM$, then we can deduce the action of the remainder generators of $\BV(\oV)$. In fact, let $(ij)\neq(12)$ and $t\in\Sn_3$ such that $t(ij)t^{-1}=(12)$. Hence
\begin{align}\label{eq:con uno es suficiente}
\yij{ij}\cdot x\mm{g}=\sgn(t)t\,\yij{12}\,t^{-1}\cdot x\mm{g}.
\end{align}
In the Appendix we give explicitly the action of $\yij{12}$ on each element $x\mm{g}$ in the basis of $\fM$. We leave this for the Appendix because it is a very long list and we don't want to bore the reader now. We shall use these computations in the next subsections without previous mention.

\subsection{The simple Verma modules}

\begin{thm}\label{teo:los simples}
The Verma modules $\fM{(e,-)}$, $\fM{(\sigma,+)}$, $\fM(\tau,1)$ and $\fM(\tau,2)$ are $\D$-simple. Therefore  $\fL(g,\varrho)\simeq\fS(g,\varrho)$ holds for these weights.
\end{thm}

\begin{proof}
By Corollary \ref{cor:verma simple xtop}, it is enough to check that $\ytop(\xtop v)\neq 0$ for some $v\in M(g,\varrho)$. In the case $\fM{(e,-)}$, using the calculations done in the appendix we have that $\ytop(\xtop \mm{e})=-12\mm{e}\neq 0$. For $\fM{(\sigma,+)}$, in the same way we obtain $\ytop(\xtop \mm{\sigma\tau})=2\mm{\sigma\tau}\neq 0$. Finally, in both cases $\fM{(\tau, 1)}$ and $\fM{(\tau, 2)}$ we have that $\ytop(\xtop \mm{\tau^{-1}})=-3\mm{\tau^{-1}}\neq 0$.

Therefore the isomorphism $\fL(g,\varrho)\simeq\fS(g,\varrho)$ holds because $\BV^{4}(V)\simeq M(e,+)$.
\end{proof}

\subsection{The Verma module \texorpdfstring{$\fM{(\sigma,-)}$}{msigmamenos}}

The aim of this subsection is to prove the next theorem which describes the submodules of $\fM{(\sigma,-)}$. As a byproduct, we find $\fL{(\sigma,-)}$ and the remaining simple $\D$-modules as subquotients of  $\fM{(\sigma,-)}$.

\begin{figure}[h]
\begin{center}
\includegraphics{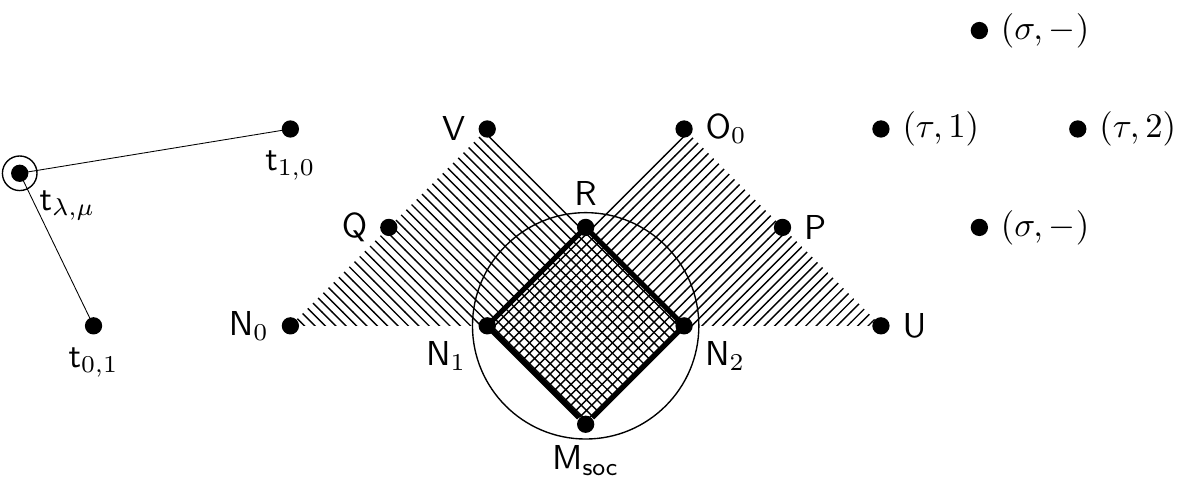}
\caption{Submodules of $\fM(\sigma,-)$}\label{fig:sigma menos}
\end{center}
\end{figure}

\begin{thm}\label{teo:el latice de sigma menos}
Every $\D$-submodule of $\fM{(\sigma,-)}$ is obtained by adding the $\D$-submodules
$$
\fS(\sigma,-),\quad\fO\fP\fU,\quad\fV\fQ\fN\quad\mbox{and}\quad\fT_{\lambda,\mu}\quad\mbox{with}\quad \lambda,\mu\in\ku.
$$
In particular, $\fX{(\sigma,-)}=\fO\fP\fU+\fV\fQ\fN+\fT_{1,0}+\fT_{0,1}$ and
$$
\fL(\sigma,-)\simeq M(\sigma,-)\oplus M(\tau,1)\oplus M(\tau,2)\oplus M(\sigma,-).
$$
as $\DSn$-modules. Moreover, $\bigl\{\fL(\sigma,-),\fL(\tau,0),\fL(e,\rho),\fL(e,+),\fL(e,+),\fL(\sigma,-)\bigr\}$ are the composition factors of $\fM(\sigma,-)$.
\end{thm}

The submodules mentioned in the statement will be given in the successive lemmas but the Figure \ref{fig:sigma menos} helps to visualise them. Namely, each dot represents a weight of $\fM(\sigma,-)$. This follows from the fusion rules given in \S\ref{subsub:fusion rules} since $\fM(\sigma,-)=\BV(V)\ot M(\sigma,-)$: $\fO_0$ is of weight $(\tau,0)$, $U$ and $V$ are of weight $(e,\rho)$, $P$ and $Q$ are of weight $(\sigma,+)$, $R$ and $\fMsoc$ are of weight $(\sigma,-)$, $\ft_{\lambda,\mu}$ is of weight $(e,+)$ for any $\lambda,\mu\in\ku$, $\fN_{\ell}$ is of weight $(\tau,\ell)$ for $\ell=0,1,2$. Of course, the weights in the same row are in the same homogeneous component and the degree decreases from the top to the bottom.

Then, the dots connected by a thick line represent the socle $\fS(\sigma,-)$, Lemma \ref{le:socle sigma menos}. The dots over the northwest lines form $\fV\fQ\fN$, Lemma \ref{le: VQN}, and those over the northeast lines form $\fO\fP\fU$, Lemma \ref{le: UPO}. The dots enclosed by circles represent a $\D$-submodule $\fT_{\lambda,\mu}$, Lemma \ref{le:T}. The isolated dots on the right hand side represent the weights of $\fL(\sigma,-)$.

We start by calculating the socle of the Verma module and then we describe the highest-weight submodules. For that, we shall use the algorithm described in Remark \ref{obs:description submodule gen by a submodule 2}. In order to avoid extra computations, we approximate the kernel of the action map before applying this algorithm using Remark \ref{obs:description submodule gen by a submodule 1}.

Let $\fN_\ell$ and $\fR$ be the $\DSn$-submodules of $\fM{(\sigma,-)}$ generated by
\begin{itemize}
\item $\fn_\ell=\zeta^{\ell}\xij{13}\xij{12}\xij{23}\mm{23}+\zeta^{-\ell}\xij{12}\xij{13}\xij{23}\mm{12}-\xij{12}\xij{13}\xij{12}\mm{13}$, for $\ell=0,1,2$, and
\smallskip
\item $\fr=-(\xij{12}\xij{13}+\xij{13}\xij{23})\mm{12}-(\xij{12}\xij{23}+\xij{13}\xij{12})\mm{13}$,
\end{itemize}
respectively. By \eqref{eq:tau en sigma sigma}, $\fn_\ell$ identifies with the element $\mm{\tau}_\ell$ and belongs to the submodule of weight $(\tau,\ell)$. Hence $\fN_\ell\simeq M(\tau,\ell)$ for $\ell=0,1,2$.

\begin{lema}\label{le:socle sigma menos}
The socle of $\fM(\sigma,-)$ is
$$
\fS(\sigma,-)=\fMsoc\oplus\fN_1\oplus\fN_2\oplus\fR,
$$
where $\fMsoc\simeq M(\sigma,-)\simeq\fR$ and $\fN_\ell\simeq M(\tau,\ell)$ as $\DSn$-modules, $\ell=1,2$.

Moreover, $\fS(\sigma,-)$ is a highest-weight module of weight $(\sigma,-)$ and therefore
$$
\fS(\sigma,-)\simeq\fL(\sigma,-).
$$
\end{lema}

\begin{proof}
We use the algorithm proposed in Remark \ref{obs:description submodule gen by a submodule 2} to compute the socle. Recall that $\fS(\sigma,-)=\BV(\oV)\fMsoc$ by Theorem \ref{teo:socle simple}.

As $\DSn$-module, we have that
$$
\oV\ot\,\fMsoc=\ku\{\yij{ij}\ot\,\xtop\mm{ij}\mid i\neq j\}\oplus\ku\{\yij{ij}\ot\,\xtop\mm{jk}\mid i\neq j\neq k\neq i\}
$$
by \eqref{eq:first decomposition}. Since $\yij{12}(\xtop\mm{12})=0$, the action map is zero in the first submodule by Remark \ref{obs:description submodule gen by a submodule 1}. The second submodule decomposes into the direct sum $M(\tau,0)\oplus M(\tau,1)\oplus M(\tau,2)$ where $(\zeta^\ell+\zeta^{-\ell}\tau^{-1}+\tau)(\yij{12}\ot\,\xtop\mm{23})$ belongs to the submodule of weight $(\tau,\ell)$ by \eqref{eq:tau en sigma sigma}. The action map applied to these elements gives
\begin{align*}
(\zeta^\ell+\zeta^{-\ell}\tau^{-1}+\tau)\yij{12}(\xtop\mm{23})&=(\zeta^\ell+\zeta^{-\ell}\tau^{-1}+\tau)(-\xij{12}\xij{13}\xij{23}\mm{12}+\xij{13}\xij{12}\xij{23}\mm{23})\\
&=(\zeta^{-\ell}-\zeta^{\ell})\xij{12}\xij{13}\xij{23}\mm{12}+(\zeta^{\ell}-1)\xij{13}\xij{12}\xij{23}\mm{23}
+(\zeta^{-\ell}-1)\xij{12}\xij{13}\xij{12}\mm{13},
\end{align*}
which is zero iff $\ell=0$. Otherwise, we obtain $\frac{1}{1-\zeta^{-\ell}}\fn_\ell$  and hence
$$
\oV\fMsoc=\fN_1\oplus \fN_2,
$$
recall Remark \ref{obs:description submodule gen by a submodule 2}.

Now, we calculate $\oV\fN_\ell$. By \eqref{eq:sigma tensor tau}, $\oV\ot\fN_\ell\simeq M(\sigma,+)\oplus M(\sigma,-)$ as $\DSn$-modules and the element $\zeta^{\ell}(1\pm(23))\yij{12}\ot \fn_\ell$ belongs to the submodule of weight $(\sigma,\pm)$. We apply the action map to these elements and obtain
\begin{align*}
\fr_\ell^{\pm}:=\zeta^{\ell}(1\pm(23))\yij{12}\fn_\ell&=\zeta^{\ell}(1\pm(23))(\xij{23}\xij{12}\mm{12}+\xij{23}\xij{13}\mm{13})=\zeta^{\ell}(1\mp1)(\xij{23}\xij{12}\mm{12}+\xij{23}\xij{13}\mm{13}).
\end{align*}
Hence $\fr_\ell^{+}=0$ and $0\neq\frac{1}{\zeta^{\ell}2}\fr_\ell^{-}=\fr\in\fR[(23)]$. Therefore
$$
\fR=\oV\fN_1=\oV\fN_2\simeq M(\sigma,-)\quad\mbox{for $\ell=1,2.$}
$$

Finally, we consider $\yij{12}\ot(13)\fr\in\oV[(12)]\ot \fR[(12)]$ and $\yij{12}\ot\fr\in\oV[(12)]\ot \fR[(23)]$. Since $\yij{12}(13)\fr=0=\yij{12}\fr$, we see that $\oV\fR=0$ by Remark \ref{obs:description submodule gen by a submodule 1}. Therefore
$$
\fS(\sigma,-)=\BV(\oV)\fMsoc=\fMsoc\oplus\fN_1\oplus\fN_2\oplus\fR.
$$
Moreover, $\fS(\sigma,-)$ is simple and generated by $\fR\simeq M(\sigma,-)$ with $\oV\fR=0$. Therefore $\fS(\sigma,-)=\fL(\sigma,-)$.
\end{proof}

Let $\fV$ and $\fQ$ be the $\DSn$-submodules of $\fM(\sigma,-)$ generated by
\begin{itemize}
\item $\fv=\zeta^{-1}\xij{23}\mm{23}+\zeta\xij{13}\mm{13}+\xij{12}\mm{12}$ and
\smallskip
\item $\fq=\xij{12}\xij{23}\mm{23}-\xij{12}\xij{13}\mm{13}$,
\end{itemize}
respectively. Note that $\fV$ is of weight $(e,\rho)$ by \eqref{eq:rho en sigma sigma}.

\begin{lema}\label{le: VQN}
Let $\fV\fQ\fN=\fS(\sigma,-)\oplus\fV\oplus\fQ\oplus\fN_0$. Then
\begin{enumerate}[(i)]
\item\label{le: VQN i} $\fV\fQ\fN=\D\fV$ is a highest-weight submodule of weight $(e,\rho)$.
\smallskip
\item\label{le: VQN ii}  $\fV\fQ\fN=\D\fQ=\D\fN_0$.
\smallskip
\item\label{le: VQN iii} $\fQ$ is of weight $(\sigma,+)$.
\end{enumerate}
\end{lema}

\begin{proof}
Since $\yij{12}\fv=0$, $\oV\fV=0$ by Remark \ref{obs:description submodule gen by a submodule 1} and thus $\D\fV$ is a highest-weight module. Hence we will use Remark \ref{obs:description submodule gen by a submodule 2} to compute $\D\fV=\BV(V)\fV$.

By \eqref{eq:sigma tensor rho}, $V\ot\fV\simeq M(\sigma,+)\oplus M(\sigma,-)$ where $(1\pm\sigma)\xij{12}\ot\fv$ belongs to the submodule of weight $(\sigma,\pm)$. The action map applied to these elements gives
$$
(1+\sigma)\xij{12}\fv=(\zeta^{-1}-\zeta)\fq\quad\mbox{and}\quad(1-\sigma)\xij{12}\fv=(13)\fr
$$
Hence
$V\fV=\fR\oplus\fQ$. In particular, $\fQ\simeq M(\sigma,+)$ which proves \eqref{le: VQN iii}.

Now, we just have to compute $V\fQ$ since $V\fR\subset\fS(\sigma,-)$. As $\D(\Sn_3)$-modules, $V\ot\fQ\simeq M(e,-)\oplus M(e,\rho)\oplus\bigoplus_{\ell=0,1,2}M(\tau,\ell)$, cf. \S\ref{subsub:fusion rules}. The action map on the components of weight $(e,-)$ and $(e,\rho)$ is zero since $\xij{12}\fq=0$, recall Remark \ref{obs:description submodule gen by a submodule 1}. Meanwhile, the action map on the components of weight $(\tau,\ell)$ is not zero. In fact, $(\zeta^\ell+\zeta^{-\ell}\tau^{-1}+\tau)(\xij{12}\ot(13)\fq)$ belongs to the submodule of weight $(\tau,\ell)$ by \eqref{eq:tau en sigma sigma} and
\begin{align*}
(\zeta^\ell+\zeta^{-\ell}\tau^{-1}+\tau)\xij{12}(13)\fq&=(\zeta^\ell+\zeta^{-\ell}\tau^{-1}+\tau)(1-\zeta^2)\left(\xij{12}\xij{13}\xij{23}\mm{12}-\xij{12}\xij{13}\xij{12}\mm{13}\right)\\
&=(\zeta^2-1)\left(\zeta^{-\ell}\xij{12}\xij{13}\xij{23}\mm{12}-\xij{12}\xij{13}\xij{12}\mm{13}
+\zeta^{\ell}\xij{13}\xij{12}\xij{23}\mm{23}\right).
\end{align*}
Hence $V\fQ=\fN_0\oplus\fN_1\oplus\fN_2$. Since $V\fN_\ell\subseteq\fMsoc$, we conclude that $\D\fV=\fV\fQ\fN$ and \eqref{le: VQN i} follows.

Finally, we proof \eqref{le: VQN ii} by noting that $\fV\subset\D\fQ$ and $\fV\subset\D\fN_0$ since
$$
\frac{1}{2}\yij{12}\fq=\frac{1}{3}\yij{12}\yij{13}\fn_0=(\zeta^{-1}-\zeta)^{-1}(1-\sigma)\fv\in\fV.
$$
\end{proof}

Using the characterization of the simple modules given in Theorem \ref{teo:bi con L}, the next result follows directly from the above lemma.

\begin{Cor}\label{cor: l e rho cociente de VQN}
The quotient $\fV\fQ\fN/\fS(\sigma,-)$ is a simple highest-weight module of weight $(e,\rho)$. Therefore
$$\fL(e,\rho)\simeq M(e,\rho)\oplus M(\sigma,+)\oplus M(\tau,0),$$
as $\DSn$-modules. \qed
\end{Cor}

Let $\fU$, $\fP$ and $\fO_\ell$ be the $\DSn$-submodules generated by
\begin{itemize}
\item $\fu=-\zeta^{-1}\xij{12}\xij{13}\xij{12}\mm{23}+\zeta\xij{12}\xij{13}\xij{23}\mm{13}+\xij{13}\xij{12}\xij{23}\mm{12}$,
\smallskip
\item $\fp=-(2\xij{13}\xij{12}+\xij{12}\xij{23})\mm{23}-(2\xij{13}\xij{23}+\xij{12}\xij{13})\mm{13}$ and
\smallskip
\item $\fo_\ell=\zeta^{-\ell}\xij{13}\mm{12}+\zeta^{\ell}\xij{12}\mm{23}+\xij{23}\mm{13}$, for $\ell=0,1,2$,
\end{itemize}
respectively. We see that $\fU$ is of weight $(e,\rho)$ by \eqref{eq:rho en sigma sigma} and Lemma \ref{descomposicion de BV as DS3mod}, and $\fO_\ell$ is of weight $(\tau,\ell)$ by \eqref{eq:tau en sigma sigma}.

\begin{lema}\label{le: UPO}
Let $\fO\fP\fU=\fS(\sigma,-)\oplus\fU\oplus\fP\oplus\fO_0$. Then
\begin{enumerate}[(i)]
\item\label{le: UPO i} $\fO\fP\fU=\D\fO_0$ is a highest-weight submodule of weight $(\tau,0)$.
\smallskip
\item\label{le: UPO ii}  $\fO\fP\fU=\D\fU=\D\fP$.
\smallskip
\item\label{le: UPO iii} $\fP$ is of weight $(\sigma,+)$.
\end{enumerate}
\end{lema}

\begin{proof}
We have that $\D\fO_0$ is a highest-weight module, since $\yij{12}\fo_0=0$ and Remark \ref{obs:description submodule gen by a submodule 1} provides $\oV\fO_0=0$. Then we compute $\D\fO_0=\BV(V)\fO_0$.

By \eqref{eq:sigma tensor tau}, $V\ot\fO_0\simeq M(\sigma,+)\oplus M(\sigma,-)$ where $(1\pm\sigma)\xij{13}\ot\fo_0$ belongs to the submodule of weight $(\sigma,\pm)$. The action map gives
$$
(1+\sigma)\xij{13}\fo_0=-\fp\quad\mbox{and}\quad(1-\sigma)\xij{13}\fo_0=(13)\fr
$$
and we obtain $V\fO_0=\fR\oplus\fP$. In particular, $\fP\simeq M(\sigma,+)$ which proves \eqref{le: VQN iii}.

Now, we calculate $V\fP$. As $\D(\Sn_3)$-modules, $V\ot\fP\simeq M(e,-)\oplus M(e,\rho)\oplus\bigoplus_{\ell=0,1,2}M(\tau,\ell)$, cf. \S\ref{subsub:fusion rules}. Using \eqref{eq:tau en sigma sigma} we see that the action map on the components of weight $(e,-)$ and $(\tau,0)$ is zero since $\xij{12}\fp+\xij{13}(23)\fp+\xij{23}(13)\fp=0$ and $(1+\tau+\tau^{-1})\xij{12}(13)\fp=0$. Meanwhile, from \eqref{eq:rho en sigma sigma} and \eqref{eq:tau en sigma sigma}, the action map on the components of weight $(\tau,1)$, $(\tau,2)$ and $(e,\rho)$ is not zero. In fact, for  $\ell=1,2$:
\begin{align*}
(\zeta^\ell+\zeta^{-\ell}\tau^{-1}+\tau)\xij{12}(13)\fp&=(\zeta^\ell+\zeta^{-\ell}\tau^{-1}+\tau)\left(\xij{12}\xij{13}\xij{23}\mm{12}+\xij{12}\xij{13}\xij{12}\mm{13}\right)\\
&=(\zeta^{\ell}-1)\xij{12}\xij{13}\xij{23}\mm{12}+(\zeta^{\ell}-\zeta^{-\ell})\xij{12}\xij{13}\xij{12}\mm{13}+(1-\zeta^{-\ell})\xij{13}\xij{12}\xij{23}\mm{23}.
\end{align*}
For the component $(e,\rho)$ we have that $\xij{12}\fp+\zeta\xij{13}(23)\fp+\zeta^{-1}\xij{23}(13)\fp=2(\zeta-\zeta^{-1})\fu$.
Hence $V\fP=\fU\oplus\fN_1\oplus\fN_2$. Since $V\fN_\ell\subseteq\fMsoc$, $V\fU\subseteq\fMsoc$ and $\fR\subset\fS(\sigma,-)$, we conclude that $\D\fO_0=\fO\fP\fU$ and \eqref{le: UPO i} follows.

For the proof of \eqref{le: UPO ii} we note that $\fO_0\subset\D\fP$ and $\fO_0\subset\D\fU$ since
$$
(1+\tau+\tau^2)\yij{12}(13)\fp=\zeta(1+\tau+\tau^2)\yij{12}(13)(1+\sigma)\yij{12}\fu=2(1-\zeta)\fo_0\in\fO_0.
$$
\end{proof}

The next result is a direct consequence of the above lemma.
\begin{Cor}\label{cor: l tau cero cociente de UPO}
The quotient $\fO\fP\fU/\fS(\sigma,-)$ is a simple highest-weight module of weight $(\tau,0)$. Therefore
$$\fL(\tau,0)\simeq M(\tau,0)\oplus M(\sigma,+)\oplus M(e,\rho)$$
as $\DSn$-modules. \qed
\end{Cor}

\begin{Cor}\label{cor: l tau cero s e rho}
As $\D$-modules, $\fL(\tau,0)\simeq\fS(e,\rho)$ and $\fL(e,\rho)\simeq\fS(\tau,0)$.
\end{Cor}

\begin{proof}
Since $\fU\subset\fM^{-3}(\sigma,-)$, $V\fU\subset\fMsoc$. Hence $\fU$ is a lowest-weight in the quotient $\fO\fP\fU/\fS(\sigma,-)$. Therefore this quotient is isomorphic to $\fS(e,\rho)$ by Theorem \ref{teo:bi con socle}. The proof of the second isomorphism is similar.
\end{proof}

The $\D(\Sn_3)$-submodules of $\fM(\sigma,-)$ of weight $(e,+)$ are
\begin{itemize}
\item $\ft_{\lambda,\mu}=\ku\bigl(\lambda\mm{e}^1+\mu\mm{e}^3\bigr)$ where $\lambda,\mu\in\ku$,
\begin{align*}
\mm{e}^1=&\xij{12}\mm{12}+\xij{23}\mm{23}+\xij{13}\mm{13}\quad\mbox{and}\quad
\mm{e}^3=\xij{13}\xij{12}\xij{23}\mm{12}-\xij{12}\xij{13}\xij{12}\mm{23}+\xij{12}\xij{13}\xij{23}\mm{13}.
\end{align*}
\end{itemize}

\begin{lema}\label{le:T}
Let $\fT_{\lambda,\mu}=\ft_{\lambda,\mu}\oplus\fS(\sigma,-)$ for any $\lambda,\mu\in\ku$. Then $\D\ft_{\lambda,\mu}=\fT_{\lambda,\mu}$. In particular, $\fT_{1,0}$ is a highest-weight submodule of weight $(e,+)$.
\end{lema}

\begin{proof}
We see that $\yij{12}\mm{e}^1=0$ and $\yij{12}\mm{e}^3=-(13)\fr\in\fR$. Meanwhile, $\xij{12}\mm{e}^1=-(13)\fr$ and clearly, $\xij{12}\mm{e}^3\in\fMsoc$. Then $V\fT_{\lambda,\mu},\oV\fT_{\lambda,\mu}\subset\fS(\sigma,-)$ and the lemma follows.
\end{proof}

\begin{Cor}\label{cor: l e mas cociente de T}
The quotient $\fT_{1,0}/\fS(\sigma,-)$ is a simple highest-weight module of weight $(e,+)$. Therefore
$$\fL(e,+)\simeq M(e,+)$$
as $\DSn$-modules. Moreover, $\fL(e,+)\simeq\fS(e,+)$ as $\D$-modules.
\end{Cor}

\begin{proof}
The first part follows as the above corollaries. In particular, $\fL(e,+)$ is one-dimensional and then it is also a lowest-weight module. Hence $\fL(e,+)\simeq\fS(e,+)$ holds.
\end{proof}

\begin{proof}[Proof of Theorem \ref{teo:el latice de sigma menos}]
Let $N$ be a $\D$-submodule of $\fM(\sigma,-)$. Then $N=\oplus_t S_t$ where $S_t$ is $\DSn$-simple. Since $N=\sum_t\D S_t$, it is enough to compute the $\D$-submodule generated by $S_t$ case-by-case according to the weight of $S_t$. Recall the weights of $\fM(\sigma,-)$ from Figure \ref{fig:sigma menos}.

\begin{enumerate}[(C{a}se 1)]
 \item If $S_t$ is of weight $(e,+)$, then $\D S_t=\fT_{\lambda,\mu}$ for some $\lambda,\mu\in\ku$ by Lemma \ref{le:T}.
 \item If $S_t$ is of weight $(e,\rho)$, then there is an element $\fa+\fb\in S_t$ with $\fa\in\fU$ and $\fb\in\fV$. Assume $\fa\neq0\neq\fb$, otherwise $\D S_t$ is either $\fU\fP \fO$ or $\fV\fQ\fN$ by Lemmas \ref{le: VQN} and \ref{le: UPO}. Then $\oV S_t=\oV\fU=\fR\oplus\fP$ because $\oV\fV=0$, and hence $\D\fP=\fO\fP\fU\subset\D S_t$. Thus $(\fa+\fb)-\fa=\fb\in\D S_t$ and therefore $\D S_t=\fO\fP\fU+\fV\fQ\fN$.	
 \item If $S_t$ is of weight $(\tau,0)$ or $(\sigma,+)$. Proceeding as above, we can see that $\D S_t\subseteq\fO\fP\fU+\fV\fQ\fN$.
 \item If $S_t$ is of weight $(\tau,\ell)$ with $\ell\neq0$, then there is an element $\fc+\fd\in S_t$ with $\fc\in\fN_\ell$ and $\fd\in\fO_\ell$. Moreover, $\fd\in\D S_t$ as $\fN_\ell\subseteq\fS(\sigma,-)\subseteq\D S_t$. Then either $\D S_t=\fS(\sigma,-)$, if $\fd=0$, or $\D S_t=\fM(\sigma,-)$ because $\yij{12}\fo_\ell=(\zeta^{-\ell}-\zeta^\ell)\mm{23}\in\D S_t$.
 \item If $S_t$ is of weight $(\sigma,-)$, then either $\D S_t=\fS(\sigma,-)$ or there is $0\neq\fy+\ft\in S_t$ with $\fy\in\fM^0( \sigma,-)$ and $\ft\in\fM^{-2}(\sigma,-)\setminus\fS(\sigma,-)$. If $\fy+\ft\neq0$, then $\D S_t=\fM(\sigma,-)$ by Lemma \ref{le:si esta 1 mas algo entonces es todo el verma} when $\fy\neq0$ or by noting that $\oV\ft=\fM^0( \sigma,-)$ when $\ft\neq0$.
\end{enumerate}
\end{proof}

\subsection{The Verma module \texorpdfstring{$\fM{(e,+)}$}{Merho}}

\begin{thm}\label{teo:el latice de e mas}
The proper $\D$-submodules of $\fM{(e,+)}$ are  $\fS{(e,+)}\subset\fX{(e,+)}$ where
\begin{enumerate}[(i)]
\item\label{teo: e mas max}  $\fX(e,+)=\oplus_{n<0}\fM^n(e,+)$ is a highest-weight submodule of weight $(\sigma,-)$.
\smallskip
\item\label{teo: e mas soc} $\fS{(e,+)}=\fMsoc$ is a highest-weight submodule of weight $(e,+)$.
\end{enumerate}
Therefore $\bigl\{\fL(e,+),\fL(\sigma,-),\fL(e,+)\bigr\}$ are the composition factors of $\fM(e,+)$.
\end{thm}

\begin{figure}[h]
\begin{center}
\includegraphics{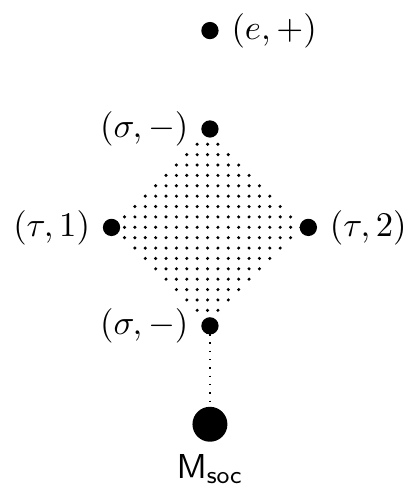}
\end{center}
\caption{Submodules of $\fM{(e,+)}$} \label{fig:e mas}
\end{figure}

In the Figure \ref{fig:e mas}, we have schemed the submodules of $\fM{(e,+)}$. The weights are represented by dots, notice that $\fM{(e,+)}\simeq\BV(V)$ as $\D(\Sn_3)$-module. The big dot in the bottom represents the socle $\fS(e,+)$. The dotted area represents the maximal $\D$-submodule $\fX(e,+)$ and hence the dot in the top corresponds to $\fL(e,+)$.

\begin{proof}
By Corollary \ref{cor: l e mas cociente de T}, $\fL(e,+)\simeq M(e,+)$ as $\D(\Sn_3)$-module and hence $\fX(e,+)=\oplus_{n<0}\fM^n(e,+)$. Since $\fX(e,+)$ is homogeneous by Theorem \ref{teo:unico submod max}, $\oV\fM^{-1}(e,+)=0$. Moreover, $\fM^{-1}(e,+)$ generates $\fX(e,+)$ because $\fM{(e,+)}\simeq\BV(V)$ as $\BV(V)$-modules. Therefore $\fX(e,+)$ is a highest-weight submodule of weight $(\sigma,-)$ and \eqref{teo: e mas max} follows.

By \eqref{teo: e mas max} and Theorem \ref{teo:bi con L}, $\fX(e,+)$ has a quotient isomorphic to $\fL(\sigma,-)$. By Theorem \ref{teo:el latice de sigma menos}, $\fL(\sigma,-)\simeq M(\sigma,-)\oplus M(\tau,1)\oplus M(\tau,2)\oplus M(\sigma,-)$ as $\D(\Sn_3)$-modules. Then, we deduce that the unique $\D$-submodule of $\fX(e,+)$ is $\fMsoc$ by inspection in the weights of the Verma module, see Figure \ref{fig:e mas}. Therefore $\fS{(e,+)}=\fMsoc$ and \eqref{teo: e mas soc} follows.

The last sentence of the statement is immediate.
\end{proof}

\subsection{The Verma module \texorpdfstring{$\fM{(\tau,0)}$}{Mtaucero}}

Let $\fJ\subset\fM^{-1}(\tau,0)$ and $\fG\subset\fM^{-2}(\tau,0)$ be the $\DSn$-submodules of weight $(\sigma,-)$ and $(e,\rho)$ with basis
\begin{itemize}
\item $\fj_i=(1-\sigma\tau^i)x_{\sigma\tau^{i+2}}\mm{\tau}$, $i=0,1,2$, see \eqref{eq:sigma tensor tau};
\item $\fg=(\xij{13}\xij{23}-\zeta^2\xij{12}\xij{13})\mm{\tau}+(\xij{13}\xij{12}-\zeta^2\xij{12}\xij{23})\mm{\tau^{-1}}$ and $\sigma\fg$,
\end{itemize}
recall \eqref{eq:tau tensor tau} and Lemma \ref{descomposicion de BV as DS3mod} \eqref{item:descomposion de BV tau}.

\begin{thm}\label{teo:el latice de tau cero}
The proper $\D$-submodules of $\fM{(\tau,0)}$ are  $\fS{(\tau,0)}\subset\fX{(\tau,0)}$ where
\begin{enumerate}[(i)]
\item\label{teo: tau cero mas} $\fX(\tau,0)=\D\fJ$ is a highest-weight submodule of weight $(\sigma,-)$.
\item\label{teo: tau cero soc} $\fS{(\tau,0)}=\D\fG$ is a highest-weight submodule of weight $(e,\rho)$.
\end{enumerate}
Therefore $\bigl\{\fL(\tau,0),\fL(\sigma,-),\fL(e,\rho)\bigr\}$ are the composition factors of $\fM(\tau,0)$.
\end{thm}

The weights of $\fM{(\tau,0)}$ are represented by dots in the Figure \ref{fig:tau cero} which can be computed using the fusion rules in \S\ref{subsub:fusion rules}. The weights conected by a line form the socle $\fS(\tau,0)$ and those over the dotted area form the maximal $\D$-submodule $\fX(\tau,0)$. The weights on the left hand side correspond to $\fL(\tau,0)$.

\begin{figure}[h]
\begin{center}
\includegraphics{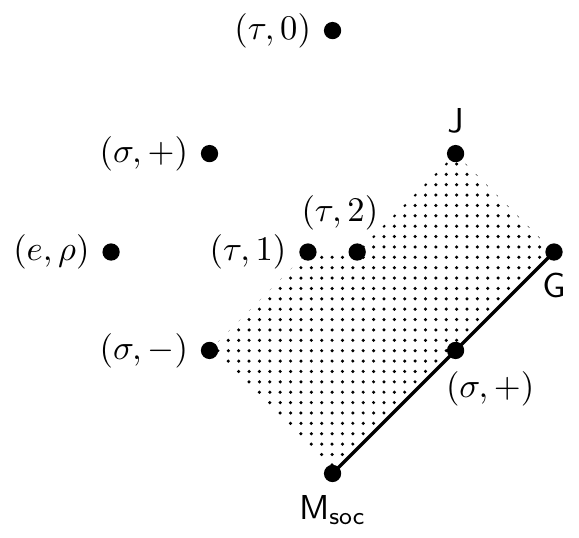}
\caption{Submodules of $\fM(\tau,0)$}\label{fig:tau cero}
\end{center}
\end{figure}

\begin{proof}
\eqref{teo: tau cero soc} By Corollary \ref{cor: l tau cero s e rho}, $\fS{(\tau,0)}$ is a highest-weight module of weight $(e,\rho)$. We see that $\oV\fG=0$ using Remark \ref{obs:description submodule gen by a submodule 1}. Therefore $\fS{(\tau,0)}=\D\fG$.

\eqref{teo: tau cero mas} By Corollary \ref{cor: l tau cero cociente de UPO}, $\fL(\tau,0)\simeq M(\tau,0)\oplus M(\sigma,+)\oplus M(e,\rho)$ as $\D(\Sn_3)$-module. Hence the sum of the simple $\D(\Sn_3)$-submodules over the dotted area in Figure \ref{fig:tau cero} have to form the maximal $\D$-submodule $\fX(\tau,0)$. Since $\fX(\tau,0)$ is homogeneous, $\oV\fJ=0$ and then $\D\fJ$ is a highest-weight submodule of weight $(\sigma,-)$. On the other hand, $\D\fJ$ has a quotient isomorphic to $\fL(\sigma,-)$ and contains the socle $\fS{(\tau,0)}$. Therefore $\fX(\tau,0)=\D\fJ$ and $\fS{(\tau,0)}$ is the unique $\D$-submodule of $\fX(\tau,0)$.
\end{proof}

\subsection{The Verma module \texorpdfstring{$\fM{(e,\rho)}$}{Merho}}
The proof of the next theorem and the description of Figure \ref{fig:e rho} are similar to the above subsection. Let $\fE\subset\fM^{-1}(e,\rho)$ and $\fC\subset\fM^{-2}(e,\rho)$ be the $\DSn$-submodules of weight $(\sigma,-)$ and $(\tau,0)$ with basis
\begin{itemize}
\item $\fe_i=\zeta^i(1-\sigma\tau^i)x_{\sigma\tau^{i+2}}\mm{\tau}$, $i=0,1,2$, see \eqref{eq:sigma tensor rho};
\item $\fc=\left(\zeta\xij{13}\xij{12}-\xij{12}\xij{23}\right)\mm{\tau}_\rho+\left(\xij{12}\xij{23}
-\zeta^{-1}\xij{13}\xij{12}\right)\mm{\tau^{-1}}_\rho$ and $\sigma\fc$,
\end{itemize}
recall \eqref{eq:tau en tau rho} and Lemma \ref{descomposicion de BV as DS3mod} \eqref{item:descomposion de BV tau}.

\begin{thm}\label{teo:el latice de erho}
The proper $\D$-submodules of $\fM(e,\rho)$ are  $\fS{(e,\rho)}\subset\fX{(e,\rho)}$ where
\begin{enumerate}[(i)]
\item\label{teo: e rho max}  $\fX(e,\rho)=\D\fE$ is a highest-weight submodule of weight $(\sigma,-)$.
\item\label{teo: e rho soc} $\fS{(e,\rho)}=\D\fC$ is a highest-weight submodule of weight $(\tau,0)$.
\end{enumerate}
Therefore $\bigl\{\fL(e,\rho),\fL(\sigma,-),\fL(\tau,0)\bigr\}$ are the composition factors of $\fM(e,\rho)$.
\qed
\end{thm}

\begin{figure}[h]
\begin{center}
\includegraphics{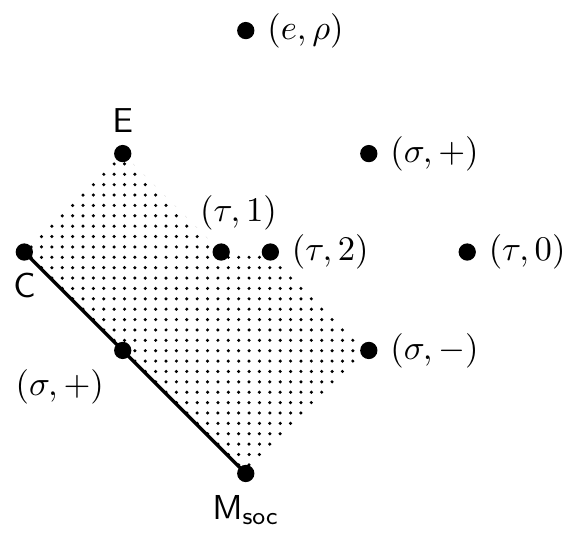}
\caption{Submodules of $\fM(e,\rho)$}\label{fig:e rho}
\end{center}
\end{figure}

\section*{Appendix}

Here we compute the action of $\yij{12}\in\D$ on the Verma Modules. We noticed in \eqref{eq:con uno es suficiente} that it suffices to calculate the action of $\yij{12}$ to know the action of the generators $\yij{23}$ and $\yij{13}$.

For the modules $\fM(e,\pm)$ and $\fM(e,\rho)$ we have only one list since all elements have weight $e$. For the module $\fM(\sigma,\pm)$ we have three lists (as the elements may have weight (12), (13) or (23)) and for the module $\fM(\tau,\ell)$ we have two lists (for the possible weights (123) and (132)).

\vspace{0.5cm}

List 1: Action on $\fM(e,\pm)$

$\yij{12}\cdot(\xij{12}\mm{e}_\pm)=(1\mp1)\mm{e}_\pm$

$\yij{12}\cdot(\xij{13}\mm{e}_\pm)=0$

$\yij{12}\cdot(\xij{23}\mm{e}_\pm)=0$

$\yij{12}\cdot(\xij{12}\xij{13}\mm{e}_\pm)=\xij{13}\mm{e}_\pm$

$\yij{12}\cdot(\xij{12}\xij{23}\mm{e}_\pm)=\xij{23}\mm{e}_\pm$

$\yij{12}\cdot(\xij{13}\xij{12}\mm{e}_\pm)=\mp\xij{23}\mm{e}_\pm$

$\yij{12}\cdot(\xij{13}\xij{23}\mm{e}_\pm)=-\xij{13}(1\mp 1)\mm{e}_\pm$

$\yij{12}\cdot(\xij{12}\xij{13}\xij{12}\mm{e}_\pm)=\xij{13}\xij{12}\mm{e}_\pm\pm\xij{12}\xij{23}\mm{e}_\pm$

$\yij{12}\cdot(\xij{12}\xij{13}\xij{23}\mm{e}_\pm)=\xij{12}\xij{13}(1\mp 1))\mm{e}_\pm+\xij{13}\xij{23}\mm{e}_\pm$

$\yij{12}\cdot(\xij{13}\xij{12}\xij{23}\mm{e}_\pm)=0$

$\yij{12}\cdot(\xij{12}\xij{13}\xij{12}\xij{23}\mm{e}_\pm)=\xij{13}\xij{12}\xij{23}(1\mp 1)\mm{e}_\pm$

\vspace{0.5cm}

List 2: Action on $\fM(e,\rho)$

$\yij{12}\cdot(\xij{12}\mm{\tau^{\pm1}}_\rho)=1\mm{\tau^{\pm1}}_\rho-1\mm{\tau^{\mp1}}_\rho$

$\yij{12}\cdot(\xij{13}\mm{\tau^{\pm1}}_\rho)=0$

$\yij{12}\cdot(\xij{23}\mm{\tau^{\pm1}}_\rho)=0$

$\yij{12}\cdot(\xij{12}\xij{13}\mm{\tau^{\pm1}}_\rho)=\xij{13}\mm{\tau^{\pm1}}_\rho$

$\yij{12}\cdot(\xij{12}\xij{23}\mm{\tau^{\pm1}}_\rho)=\xij{23}\mm{\tau^{\pm1}}_\rho$

$\yij{12}\cdot(\xij{13}\xij{12}\mm{\tau^{\pm1}}_\rho)=-{\zeta}^{\pm 1}\xij{23}\mm{\tau^{\mp1}}_\rho$

$\yij{12}\cdot(\xij{13}\xij{23}\mm{\tau^{\pm1}}_\rho)=-\xij{13}\mm{\tau^{\pm1}}_\rho+{\zeta}^{\pm1}\xij{13}\mm{\tau^{\mp1}}_\rho$

$\yij{12}\cdot(\xij{12}\xij{13}\xij{12}\mm{\tau^{\pm1}}_\rho)=\xij{13}\xij{12}\mm{\tau^{\pm1}}_\rho+{\zeta}^{\mp1}\xij{12}\xij{23}\mm{\tau^{\mp1}}_\rho$

$\yij{12}\cdot(\xij{12}\xij{13}\xij{23}\mm{\tau^{\pm1}}_\rho)=\xij{12}\xij{13}\mm{\tau^{\pm1}}_\rho-{\zeta}^{\pm1}\xij{12}\xij{13}\mm{\tau^{\mp1}}_\rho+\xij{13}\xij{23}\mm{\tau^{\pm1}}_\rho$

$\yij{12}\cdot(\xij{13}\xij{12}\xij{23}\mm{\tau^{\pm1}}_\rho)=0$

$\yij{12}\cdot(\xij{12}\xij{13}\xij{12}\xij{23}\mm{\tau^{\pm1}}_\rho)=\xij{13}\xij{12}\xij{23}\mm{\tau^{\pm1}}_\rho-\xij{13}\xij{12}\xij{23}\mm{\tau^{\mp1}}_\rho$
\vspace{0.5cm}

List 3: Action on $\fM(\sigma,\pm)$

$\yij{12}\cdot(\xij{12}\mm{12}_{\pm})=(1\pm1)\mm{12}_{\pm}$

$\yij{12}\cdot(\xij{13}\mm{12}_\pm)=\pm1\mm{23}_\pm$

$\yij{12}\cdot(\xij{23}\mm{12}_{\pm})=\pm1\mm{13}_\pm$

$\yij{12}\cdot(\xij{12}\xij{13}\mm{12}_{\pm})=\xij{13}\mm{12}_{\pm}\mp\xij{12}\mm{23}_{\pm}$

$\yij{12}\cdot(\xij{12}\xij{23}\mm{12}_{\pm})=\xij{23}\mm{12}_{\pm}\mp\xij{12}\mm{13}_{\pm}$

$\yij{12}\cdot(\xij{13}\xij{12}\mm{12}_{\pm})=0$

$\yij{12}\cdot(\xij{13}\xij{23}\mm{12}_{\pm})=-\xij{13}\mm{12}_{\pm}\pm\xij{12}\mm{23}_{\pm}$

$\yij{12}\cdot(\xij{12}\xij{13}\xij{12}\mm{12}_{\pm})=\xij{13}\xij{12}\mm{12}_{\pm}$

$\yij{12}\cdot(\xij{12}\xij{13}\xij{23}\mm{12}_{\pm})=\xij{13}\xij{23}\mm{12}_{\pm}+\xij{12}\xij{13}\mm{12}_{\pm}$

$\yij{12}\cdot(\xij{13}\xij{12}\xij{23}\mm{12}_{\pm})=0$

$\yij{12}\cdot(\xij{12}\xij{13}\xij{12}\xij{23}\mm{12}_{\pm})=\xij{13}\xij{12}\xij{23}(1\pm 1)\mm{12}_{\pm}$

\vspace{0.5cm}

List 4: Action on $\fM(\sigma,\pm)$

$\yij{12}\cdot(\xij{12}\mm{13}_{\pm})=1\mm{13}_{\pm}$

$\yij{12}\cdot(\xij{13}\mm{13}_{\pm})=0$

$\yij{12}\cdot(\xij{23}\mm{13}_{\pm})=0$

$\yij{12}\cdot(\xij{12}\xij{13}\mm{13}_{\pm})=\xij{13}\mm{13}_{\pm}\pm\xij{23}\mm{23}_{\pm}$

$\yij{12}\cdot(\xij{12}\xij{23}\mm{13}_{\pm})=\xij{23}\mm{13}_{\pm}$

$\yij{12}\cdot(\xij{13}\xij{12}\mm{13}_{\pm})=\pm\xij{23}\mm{13}_{\pm}$

$\yij{12}\cdot(\xij{13}\xij{23}\mm{13}_{\pm})=-\xij{13}\mm{13}_{\pm}$

$\yij{12}\cdot(\xij{12}\xij{13}\xij{12}\mm{13}_{\pm})=\xij{13}\xij{12}\mm{13}_{\pm}\mp\xij{12}\xij{23}\mm{13}_{\pm}$

$\yij{12}\cdot(\xij{12}\xij{13}\xij{23}\mm{13}_{\pm})=
(\xij{13}\xij{23}+\xij{12}\xij{13})\mm{13}_{\pm}\mp(\xij{13}\xij{12}+\xij{12}\xij{23})\mm{23}_{\pm}$

$\yij{12}\cdot(\xij{13}\xij{12}\xij{23}\mm{13}_{\pm})=\pm\xij{13}\xij{12}\mm{12}_{\pm}$

$\yij{12}\cdot(\xij{12}\xij{13}\xij{12}\xij{23}\mm{13}_{\pm})=\xij{13}\xij{12}\xij{23}\mm{13}_{\pm}\mp\xij{12}\xij{13}\xij{12}\mm{12}_{\pm}$

\vspace{0.5cm}

List 5: Action on $\fM(\sigma,\pm)$

$\yij{12}\cdot(\xij{12}\mm{23}_{\pm})=1\mm{23}_{\pm}$

$\yij{12}\cdot(\xij{13}\mm{23}_{\pm})=0$

$\yij{12}\cdot(\xij{23}\mm{23}_{\pm})=0$

$\yij{12}\cdot(\xij{12}\xij{13}\mm{23}_{\pm})=\xij{13}\mm{23}_{\pm}$

$\yij{12}\cdot(\xij{12}\xij{23}\mm{23}_{\pm})=\xij{23}\mm{23}_{\pm}\pm\xij{13}\mm{13}_{\pm}$

$\yij{12}\cdot(\xij{13}\xij{12}\mm{23}_{\pm})=\mp\xij{13}\mm{13}_{\pm}$

$\yij{12}\cdot(\xij{13}\xij{23}\mm{23}_{\pm})=-\xij{13}(1\pm 1)\mm{23}_{\pm}$

$\yij{12}\cdot(\xij{12}\xij{13}\xij{12}\mm{23}_{\pm})=\xij{13}\xij{12}\mm{23}_{\pm}\mp\xij{13}\xij{23}\mm{13}_{\pm}$

$\yij{12}\cdot(\xij{12}\xij{13}\xij{23}\mm{23}_{\pm})=\xij{13}\xij{23}\mm{23}_{\pm}-\xij{12}\xij{13}(1\pm 1)\mm{23}_{\pm}$

$\yij{12}\cdot(\xij{13}\xij{12}\xij{23}\mm{23}_{\pm})=\mp\xij{12}\xij{13}\mm{12}_{\pm}\mp\xij{13}\xij{23}\mm{12}_{\pm}$

$\yij{12}\cdot(\xij{12}\xij{13}\xij{12}\xij{23}\mm{23}_{\pm})=\pm\xij{12}\xij{13}\xij{23}\mm{12}_{\pm}+\xij{13}\xij{12}\xij{23}\mm{23}_{\pm}$

\vspace{0.5cm}

List 6: Action on $\fM(\tau,\ell)$

$\yij{12}\cdot(\xij{12}\mm{123}_\ell)=1\mm{123}_\ell$

$\yij{12}\cdot(\xij{13}\mm{123}_\ell)=0$

$\yij{12}\cdot(\xij{23}\mm{123}_\ell)=-\zeta^\ell\mm{132}_\ell$

$\yij{12}\cdot(\xij{12}\xij{13}\mm{123}_\ell)=\xij{13}\mm{123}_\ell-\xij{23}\mm{132}_\ell$

$\yij{12}\cdot(\xij{12}\xij{23}\mm{123}_\ell)=\xij{23}\mm{123}_\ell+\zeta^\ell\xij{12}\mm{132}_\ell$

$\yij{12}\cdot(\xij{13}\xij{12}\mm{123}_\ell)=0$

$\yij{12}\cdot(\xij{13}\xij{23}\mm{123}_\ell)=-\xij{13}\mm{123}_\ell$

$\yij{12}\cdot(\xij{12}\xij{13}\xij{12}\mm{123}_\ell)=\xij{13}\xij{12}\mm{123}_\ell$

$\yij{12}\cdot(\xij{12}\xij{13}\xij{23}\mm{123}_\ell)=
(\xij{13}\xij{23}+\xij{12}\xij{13})\mm{123}_\ell+(\xij{13}\xij{12}+\xij{12}\xij{23})\mm{132}_\ell$

$\yij{12}\cdot(\xij{13}\xij{12}\xij{23}\mm{123}_\ell)=\zeta^{-\ell}\xij{12}\xij{13}\mm{132}_\ell+\zeta^{-\ell}\xij{13}\xij{23}\mm{132}_\ell$

$\yij{12}\cdot(\xij{12}\xij{13}\xij{12}\xij{23}\mm{123}_\ell)=\xij{13}\xij{12}\xij{23}\mm{123}_\ell-\zeta^{-\ell}\xij{12}\xij{13}\xij{23}\mm{132}_\ell$

\vspace{0.5cm}

List 7: Action on $\fM(\tau,\ell)$

$\yij{12}\cdot(\xij{12}\mm{132}_\ell)=1\mm{132}_\ell$

$\yij{12}\cdot(\xij{13}\mm{132}_\ell)=-\zeta^\ell\mm{123}_\ell$

$\yij{12}\cdot(\xij{23}\mm{132}_\ell)=0$

$\yij{12}\cdot(\xij{12}\xij{13}\mm{132}_\ell)=\xij{13}\mm{132}_\ell+\zeta^\ell\xij{12}\mm{123}_\ell$

$\yij{12}\cdot(\xij{12}\xij{23}\mm{132}_\ell)=\xij{23}\mm{132}_\ell-\xij{13}\mm{123}_\ell$

$\yij{12}\cdot(\xij{13}\xij{12}\mm{132}_\ell)=\xij{13}\mm{123}_\ell$

$\yij{12}\cdot(\xij{13}\xij{23}\mm{132}_\ell)=\zeta^\ell\xij{12}\mm{123}_\ell-\xij{13}\mm{132}_\ell$

$\yij{12}\cdot(\xij{12}\xij{13}\xij{12}\mm{132}_\ell)=\xij{13}\xij{12}\mm{132}_\ell+\xij{13}\xij{23}\mm{123}_\ell$

$\yij{12}\cdot(\xij{12}\xij{13}\xij{23}\mm{132}_\ell)=\xij{13}\xij{23}\mm{132}_\ell+\xij{12}\xij{13}\mm{132}_\ell$

$\yij{12}\cdot(\xij{13}\xij{12}\xij{23}\mm{132}_\ell)=-\zeta^{-\ell}\xij{13}\xij{12}\mm{123}_\ell$

$\yij{12}\cdot(\xij{12}\xij{13}\xij{12}\xij{23}\mm{132}_\ell)=\xij{13}\xij{12}\xij{23}\mm{132}_\ell+\zeta^{-\ell}\xij{12}\xij{13}\xij{12}\mm{123}_\ell$

\end{document}